\documentclass[11pt]{amsart}

\usepackage[margin=2.7cm]{geometry}
\usepackage{amsmath,amssymb,amsthm}
\usepackage{booktabs}
\usepackage{array}
\usepackage{verbatim}
\usepackage[colorlinks=true,linkcolor=blue,citecolor=blue,urlcolor=blue]{hyperref}

\newtheorem{theorem}{Theorem}[section]
\newtheorem{proposition}[theorem]{Proposition}
\newtheorem{lemma}[theorem]{Lemma}
\newtheorem{corollary}[theorem]{Corollary}
\newtheorem{conjecture}[theorem]{Conjecture}
\theoremstyle{definition}

\newtheorem{remark}[theorem]{Remark}

\newtheorem{question}[theorem]{Question}

\newcommand{\F}{\mathbb{F}}
\newcommand{\Fn}{\mathbb{F}_{2^n}}
\newcommand{\Tr}{\operatorname{Tr}}
\newcommand{\MCM}{\operatorname{MCM}}
\newcommand{\supp}{\operatorname{supp}}
\newcommand{\ch}{\chi}

\title[On a conjecture on the Kasami APN function]{On a conjecture on the Kasami APN function:\\
reductions, structure theorems, a proof for $k\bmod n\in\{1,2,n{-}2,n{-}1\}$,\\
and exhaustive verification for $n\le 13$}
\author{G.~P.~Nagy}
\author{A.~Vajda}
\address{Bolyai Institute, University of Szeged, Aradi V\'{e}rtan\'{u}k tere 1, H-6720 Szeged, Hungary}
\thanks{This report documents a complete working record: the initial plan,
all calculations and derivations (including a refuted first hypothesis), full
proofs of the results obtained, and the numerical verification protocol. The
proofs were obtained by prompting the AI assistant \emph{Claude Fable 5} with
the conjecture statement, first formulated by Carlet~\cite{Carlet2018} and
subsequently posed as an open Kasami problem at NSUCRYPTO~2019~\cite{NSUC2019},
together with background hints and a pointer to the
companion repository \cite{LeanRepo}. All new results and proofs have
subsequently been formally verified in the \textsc{Lean}~4 theorem prover by
Aristotle (Harmonic); see the Acknowledgements.}
\date{August 16, 2026}

\begin{document}

\begin{abstract}
We study Carlet's cyclic-additive conjecture for the Kasami almost perfect nonlinear (APN) function
$F(x)=x^{4^k-2^k+1}$ on $\Fn$, $\gcd(k,n)=1$: for the $2^{n-1}$-element set
$\Delta=\{F(b)+F(b+1)+1: b\in\Fn\}$ and all distinct nonzero $v_1,v_2\in\Fn$,
\[
\bigl|\{(x,y,z)\in\Delta^3 : v_1x+v_2y+(v_1+v_2)z=0\}\bigr| \;=\; 2^{2n-3}.
\]
This exact triple-count condition was first formulated by Carlet in his 2018
cyclic-additive difference-set framework~\cite{Carlet2018}; the Kasami instance
was subsequently posed as an open problem at the NSUCRYPTO~2019 cryptographic
olympiad, whose individual proposer was not publicly disclosed~\cite{NSUC2019}.

We prove the conjecture for $k\bmod n\in\{1,2,n-2,n-1\}$, in particular a
complete proof for $k=2$ ($d=13$) via a quadratic-form theory and an exact
root-count reduction, and we verify it exhaustively by computer for every
admissible $(n,k)$ with $n\le13$. Along the way we obtain several exact
character-sum reductions of the conjecture, a closed rational form for the
relevant Walsh coefficients, and structural results (Gold and inverse-map
analogues, and a na\"{\i}ve support-disjointness mechanism that we show fails
in general) that clarify what a proof of the general case would require; the
detailed statements are collected at the end of Section~\ref{sec:conj}. The
general case remains open.

Every proof in this paper was obtained by the AI assistant Claude Fable~5 and
has subsequently been formally verified in the Lean theorem prover by
Aristotle (Harmonic); see \cite{KasamiRepo} and the Acknowledgements for provenance details.
\end{abstract}

\maketitle
\tableofcontents

%======================================================================
\section{The conjecture}\label{sec:conj}

Throughout, $\Fn$ is the binary field with $2^n$ elements,
$\Tr(x)=\sum_{i=0}^{n-1}x^{2^i}$ is the absolute trace,
$\ch(x)=(-1)^{\Tr(x)}$ is the canonical additive character, and
\[
H_0=\{x\in\Fn:\Tr(x)=0\},\qquad H_1=\{x\in\Fn:\Tr(x)=1\}
\]
are the trace hyperplane and its complement, $|H_0|=|H_1|=2^{n-1}$.
For a positive integer $k$ with $\gcd(k,n)=1$ we write
\[
q=2^k,\qquad d=d(k)=4^k-2^k+1=q^2-q+1,\qquad F(x)=x^{d}\quad(\text{Kasami function}),
\]
\[
\delta(b)=F(b)+F(b+1)+1,\qquad \Delta=\{\delta(b): b\in\Fn\}.
\]

\begin{conjecture}[Carlet~\cite{Carlet2018}; NSUCRYPTO 2019~\cite{NSUC2019}]\label{conj:main}
Let $k,n$ be positive integers with $\gcd(k,n)=1$. Then for every pair of
distinct nonzero $v_1,v_2\in\Fn$,
\begin{equation}\label{eq:conj}
N(v_1,v_2):=\bigl|\{(x,y,z)\in\Delta^3: v_1x+v_2y+(v_1+v_2)z=0\}\bigr|=2^{2n-3}.
\end{equation}
\end{conjecture}

Since $|\Delta|=2^{n-1}$ (Lemma~\ref{lem:2to1}) and one $\Fn$-linear condition
should cut $\Delta^3$ ($2^{3n-3}$ triples) by a factor $2^n$, the conjectured
value $2^{2n-3}$ is exactly the ``perfect equidistribution'' count; we also show
(Remark~\ref{rem:average}) that $2^{2n-3}$ is the \emph{average} of
$N$ over the admissible pairs, so the conjecture asserts that $N$ is
\emph{constant}, pinned at its mean.

The known supporting facts, listed in the original prompt and formally verified in
the \textsc{Lean} repository \cite{LeanRepo}, are recalled (and in one point
corrected) in Section~\ref{sec:facts}.

\subsection{Summary of the results}\label{sec:conj-summary}

We collect here, as a detailed roadmap, the exact statements proved in the
rest of the paper (all proofs are given in the indicated sections).

\begin{itemize}
\item \textbf{Character reduction.} The count $N(v_1,v_2)$ depends only on
$\rho=v_2/v_1$ and equals $2^{2n-3}+2^{-n}Z(\rho)$ for an explicit character
sum $Z(\rho)$ (Proposition~\ref{prop:red1}); the conjecture is equivalent to
$Z(\rho)=0$ for all $\rho\neq0,1$, equivalently to the vanishing statement
\[
\sum_{t,a\in\Fn}(-1)^{\Tr\left(\psi(t)+\psi(t+a)+\psi(t+\rho a)\right)}=0
\qquad(\rho\neq 0,1),
\]
where $\psi=\MCM^{-1}$ is the inverse of the M\"uller--Cohen--Matthews
permutation (Theorem~\ref{thm:red3}; after a Frobenius transfer one may
assume $k$ odd, in which case $\MCM$ is a permutation --- we also record the
necessary correction that $\MCM$ is \emph{not} a permutation for even $k$,
Lemma~\ref{lem:mcmparity}).
\item \textbf{Easy and monomial cases.} The conjecture holds for
$k\equiv\pm1\pmod n$ (all $n$, Theorem~\ref{thm:k1}); the analogous vanishing
statement holds for all Gold permutations $x^{2^j+1}$ ($n$ odd) and for the
inverse map $x^{-1}$ (all $n$), via an exact root-count mechanism
(Section~\ref{sec:monomial}).
\item \textbf{A closed form.} For odd $n$ we derive a new closed form
$S(a)=\ch(a)\sum_{\Tr(m)=1}\ch\bigl(a\,\Phi(m)\bigr)$,
$\Phi(m)=(m+m^{q}+m^{q^2})\,m^{-(q^2-q+1)}$ ($q=2^k$), exhibiting a second
rational parametrization $\Delta+1=\Phi(\{\Tr=1\})$ (Section~\ref{sec:closedform}).
\item \textbf{A refuted mechanism.} The na\"{\i}ve mechanism (disjointness of
the support of the character transform of $\Delta$ from zero-sum triples)
provably fails for $2\le k\le n-2$ (Section~\ref{sec:termwise}), so any proof
of the general case must exploit genuine cancellation.
\item \textbf{The case $k=2$.} Our main result is a \emph{complete proof of
the conjecture for $k=2$} ($d=13$), hence for all $k\equiv\pm2\pmod n$
(Section~\ref{sec:k2}): we develop a quadratic-form theory in which
$\MCM(x)=x(x+1)^5$ and all character sums are quadratic Gauss sums, prove a
\emph{master correspondence} identifying the radicals of the pencil with two
low-degree families (among them the Bluher--Helleseth--Kholosha family
$z^5+z+u$), deduce by a counting argument that $x(x+1)^5$ is exactly
$2$-to-$1$ with trace-split fibers and that the companion sum $G_1$ vanishes
identically, convert the triple correlation by monomial substitutions into
the exact root count $\#\{\tau:(a\tau^5+b)^3=(\tau^3+1)^5\}=1$ on the
parameter curve $a^6+b^6=1$, and finally pin this count by the explicit root
$\tau_0=a/b$ together with the unconditional average identity --- no
Arf-invariant or sign computations are required. Equivalently: on the
supersingular Fermat cubic $x^3+y^3=1$, the system $(W,T)\in E$,
$W^5+pT^5=q$ has the unique solution $(1/q,p/q)$, a translation by a
rational $3$-torsion point.
\item \textbf{Numerical verification.} We verify the conjecture by computer
\emph{exhaustively over all $(v_1,v_2)$} for every admissible $k$ for all
$n\le 13$ (Section~\ref{sec:numerics}).
\end{itemize}

The general case remains open; Section~\ref{sec:strategy} formulates the
sharpest equivalent forms and outlines strategies towards it.

%======================================================================
\section{The plan}\label{sec:plan}

This document was produced by executing the following plan, formulated at the
outset; we record it verbatim, as it also serves as a map of the paper.

\begin{enumerate}
\item \textbf{Understand and pin down the objects.}
  Read the problem statement, which includes background
  hints and a pointer to the companion \textsc{Lean} repository
  \cite{LeanRepo}; inventory that repository (it contains formally verified
  proofs of the APN and AB properties of the Kasami monomial function) to
  know exactly which facts are proved and under which hypotheses (done; one
  hypothesis --- $k$ odd for the M\"uller--Cohen--Matthews permutation ---
  turns out to be essential and is missing from the informal fact list; see
  Lemma~\ref{lem:mcmparity}).
\item \textbf{Reduce by characters.} Express $N(v_1,v_2)$ through the character
  sums $S(a)=\sum_{t\in\Delta}\ch(at)$; isolate the main term $2^{2n-3}$ and an
  error term $Z(\rho)$ depending only on $\rho=v_2/v_1$
  (Section~\ref{sec:reduction1}).
\item \textbf{Formulate a candidate mechanism.} Prove that the conjecture would
  follow if the support of $S$ contained no three distinct elements summing to
  $0$ (e.g.\ if $\supp S\setminus\{0\}\subseteq H_1$); test this numerically.
  \emph{Outcome: the mechanism holds precisely for $k\equiv\pm1\pmod n$ and
  fails otherwise; the failure is itself informative}
  (Section~\ref{sec:termwise}).
\item \textbf{Structure theory of $\Delta$ and $S$.} Use the exact $2$-to-$1$
  structure of $\delta$ and the identity
  $\delta(b)=\MCM(b^2+b)$ to obtain closed expressions for $S$;
  derive a new single-variable closed form for $S$ via a Gold-substitution
  and telescoping argument (Section~\ref{sec:closedform}).
\item \textbf{Find the sharpest equivalent form of the conjecture.} Convert the
  count over $\Delta^3$ into a full-space statement about the inverse
  permutation $\psi=\MCM^{-1}$; handle even $k$ by a Frobenius transfer
  (Section~\ref{sec:reduction3}).
\item \textbf{Prove what can be proved.} Settle $k\equiv\pm1\pmod n$
  completely; isolate the ``monomial mechanism'' (exact root counts) and prove
  the analogous statements for Gold permutations and the inverse map, to
  identify precisely what is missing in the $\MCM^{-1}$ case
  (Sections~\ref{sec:k1},~\ref{sec:monomial}).
\item \textbf{Verify massively.} Implement independent verification:
  exhaustive over all $(v_1,v_2)$ for all admissible $k$ and $n\le 13$;
  cross-validate every theoretical identity numerically
  (Section~\ref{sec:numerics}).
\item \textbf{Report honestly.} Write up all calculations, including the
  refuted hypothesis from step 3, the proved partial results, the equivalent
  reformulations, the numerical evidence, and a strategy discussion
  (Sections~\ref{sec:status},~\ref{sec:strategy}).
\end{enumerate}

\noindent
\textbf{Summary of outcomes.}
Steps 1--2 and 4--7 succeeded in full; step 3 produced a counterexample to the
candidate mechanism (an instructive negative result). The conjecture is proved
for $k\equiv\pm1\pmod n$, reduced in general to a single clean vanishing
statement for $\MCM^{-1}$, and verified exhaustively for all $n\le13$.
A second campaign (Section~\ref{sec:k2}), executed after the general
reductions, attacked the first open case $k=2$ through quadratic-form theory;
in its final stage the residual claim (``Claim $\mathrm{C}'$'') was proved by
a counting argument, the sign problem was eliminated by monomial
substitutions, and the case $k=2$ of the conjecture was \emph{proved in
full} (Theorem~\ref{thm:k2main}), extending the settled cases to
$k\bmod n\in\{1,2,n-2,n-1\}$. The general case remains open.

%======================================================================
\section{Background facts, with one correction}\label{sec:facts}

Let $T_k(x)=\sum_{i=0}^{k-1}x^{2^i}$ denote the truncated trace and
\begin{equation}\label{eq:mcmdef}
\MCM(x)\;=\;\frac{T_k(x)^{q+1}}{x^{q}}\qquad(q=2^k),
\end{equation}
the M\"uller--Cohen--Matthews (MCM) polynomial \cite{CohenMatthews,Dobbertin99};
as $T_k(x)^{q}$ is divisible by $x^{q}$, \eqref{eq:mcmdef} is a polynomial map
with $\MCM(0)=0$.

\begin{lemma}[telescoping identities]\label{lem:telescope}
For all $x\in\Fn$:
$T_k(x)^2+T_k(x)=x^{q}+x$ and $T_k(x^2+x)=x^{q}+x$.
\end{lemma}
\begin{proof}
$T_k(x)^2=\sum_{i=1}^{k}x^{2^i}$, so $T_k(x)^2+T_k(x)=x^{2^k}+x$.
Similarly $T_k(x^2+x)=\sum_{i=0}^{k-1}(x^{2^{i+1}}+x^{2^i})=x^{2^k}+x$.
\end{proof}

\begin{lemma}[the key identity; \textsc{Lean}-verified in \cite{LeanRepo}]\label{lem:identity3}
For all $b\in\Fn$ and all $k,n\ge1$:
\begin{equation}\label{eq:id3}
F(b)+F(b+1)+1=\frac{(b^{q}+b)^{q+1}}{(b^2+b)^{q}}=\MCM(b^2+b).
\end{equation}
\end{lemma}
\begin{proof}
The second equality is Lemma~\ref{lem:telescope}.
For the first, put $u=b^{q}+b$ and $s=b^2+b=b(b+1)$. If $b\in\{0,1\}$ both sides
are $0+1+1=0=\MCM(0)$. Otherwise $s\neq0$ and it suffices to prove
$u^{q+1}=\bigl(F(b)+F(b+1)+1\bigr)s^{q}$. Expand both sides ($d=q^2-q+1$):
\begin{align*}
u^{q+1}&=(b^{q^2}+b^{q})(b^{q}+b)=b^{q^2+q}+b^{q^2+1}+b^{2q}+b^{q+1};\\
F(b)\,s^{q}&=b^{q^2-q+1}\,(b^{2q}+b^{q})=b^{q^2+q+1}+b^{q^2+1};\\
F(b+1)\,s^{q}&=(b+1)^{q^2-q+1}\,b^{q}(b+1)^{q}
=b^{q}(b+1)^{q^2+1}=b^{q}(b^{q^2}+1)(b+1)\\
&=b^{q^2+q+1}+b^{q^2+q}+b^{q+1}+b^{q};\\
1\cdot s^{q}&=b^{2q}+b^{q}.
\end{align*}
Adding the last three lines (char.\ $2$) gives
$b^{q^2+1}+b^{q^2+q}+b^{q+1}+b^{2q}=u^{q+1}$.
\end{proof}

\begin{lemma}[exact $2$-to-$1$ structure]\label{lem:2to1}
Let $\gcd(k,n)=1$. Every fiber of $\delta(b)=F(b)+F(b+1)+1$ is exactly a pair
$\{b,b+1\}$; consequently $|\Delta|=2^{n-1}$, $0=\delta(0)\in\Delta$, and for
every $a\in\Fn$,
\begin{equation}\label{eq:2to1}
\sum_{b\in\Fn}\ch\bigl(a\,\delta(b)\bigr)=2S(a),\qquad
S(a):=\sum_{t\in\Delta}\ch(at).
\end{equation}
Moreover $b\mapsto b^2+b$ is exactly $2$-to-$1$ from $\Fn$ onto $H_0$
(fibers $\{b,b+1\}$), so by \eqref{eq:id3},
$\Delta=\MCM(H_0)$ and $\MCM$ restricted to $H_0$ is injective.
\end{lemma}
\begin{proof}
$\delta(b)=\delta(b+1)$ is clear. The Kasami function is APN for
$\gcd(k,n)=1$ \cite{Kasami71,Dobbertin99} (formally verified without any parity
hypotheses in \cite{LeanRepo}), i.e.\ $F(x)+F(x+1)=c$ has at most two solutions
for every $c$; hence the fibers are exactly the pairs $\{b,b+1\}$.
The map $b\mapsto b^2+b$ is $2$-to-$1$ with image contained in $H_0$
(as $\Tr(b^2)=\Tr(b)$) and image size $2^{n-1}=|H_0|$.
Injectivity of $\MCM$ on $H_0$ follows since
$|\MCM(H_0)|=|\Delta|=2^{n-1}=|H_0|$.
\end{proof}

The next elementary observation corrects the fact list of the original prompt,
which asserts that $\MCM$ is a permutation polynomial whenever $\gcd(k,n)=1$.

\begin{lemma}[parity caveat for the MCM permutation]\label{lem:mcmparity}
$T_k(1)=k\bmod 2$ and hence $\MCM(1)=k\bmod 2$. Consequently, if $k$ is even
then $\MCM(1)=0=\MCM(0)$ and $\MCM$ is \emph{not} a permutation of $\Fn$.
If $k$ is odd (and $\gcd(k,n)=1$, $1\le k<n$), then $\MCM$ \emph{is} a
permutation of $\Fn$, for every parity of $n$
\cite{CohenMatthews,Dobbertin99}; this is the statement formally verified in
\cite{LeanRepo} (\texttt{mcm\_isPermutation}, hypotheses $0<k<n$,
$\gcd(k,n)=1$, $k$ odd, no hypothesis on the parity of $n$).
Note that if $k$ is even and $\gcd(k,n)=1$ then $n$ is odd, so $\Tr(1)=1$ and
$1\notin H_0$: the collision $\MCM(0)=\MCM(1)$ does not conflict with the
injectivity of $\MCM$ on $H_0$ (Lemma~\ref{lem:2to1}).
\end{lemma}
\begin{proof}
$T_k(1)=\underbrace{1+\cdots+1}_{k}=k\bmod2$, and
$\MCM(1)=T_k(1)^{q+1}/1$. The permutation statement for odd $k$ is
\cite{CohenMatthews}; see also \cite{Dobbertin99} and \cite{LeanRepo}.
\end{proof}

Finally we record the Frobenius equivalence between $k$ and $n-k$, which we use
to reduce to odd $k$. Note that since $2^{k}$ only enters through its residue
class, $d(k)$ depends only on $k\bmod n$ as an exponent on $\Fn^{\ast}$, so we
may and do assume $1\le k\le n-1$.

\begin{lemma}[Frobenius transfer $k\leftrightarrow n-k$]\label{lem:frobenius}
Let $\gcd(k,n)=1$, $1\le k\le n-1$. Then
$2^{2k}\,d(n-k)\equiv d(k)\pmod{2^n-1}$; hence
$x^{d(k)}=\bigl(x^{d(n-k)}\bigr)^{2^{2k}}$ for all $x\in\Fn$,
\[
\delta_k(b)=\bigl(\delta_{n-k}(b)\bigr)^{2^{2k}},\qquad
\Delta_k=\bigl(\Delta_{n-k}\bigr)^{2^{2k}}\quad(\text{setwise Frobenius image}),
\]
and Conjecture~\ref{conj:main} holds for $(n,k)$ if and only if it holds for
$(n,n-k)$. In particular, since $\gcd(k,n)=1$ forces $k$ or $n-k$ to be odd,
it suffices to prove the conjecture for odd $k$.
\end{lemma}
\begin{proof}
With $q=2^k$: $2^{2k}d(n-k)=2^{2k}(2^{2n-2k}-2^{n-k}+1)
=2^{2n}-2^{n+k}+2^{2k}\equiv 1-2^{k}+2^{2k}=d(k) \pmod{2^n-1}$,
because $2^{2n}\equiv1$ and $2^{n+k}\equiv2^k$.
Applying the field automorphism $\pi(x)=x^{2^{2k}}$:
$\delta_k(b)=\pi(\delta_{n-k}(b))$ since $\pi$ is additive and fixes $1$.
Hence $\Delta_k=\pi(\Delta_{n-k})$. If \eqref{eq:conj} holds for $\Delta_{n-k}$
and all pairs, then for any distinct nonzero $v_1,v_2$,
\[
\{(x,y,z)\in\Delta_k^3: v_1x+v_2y+(v_1{+}v_2)z=0\}
=\pi^{\times3}\Bigl(\{(x',y',z')\in\Delta_{n-k}^3:
v_1'x'+v_2'y'+(v_1'{+}v_2')z'=0\}\Bigr)
\]
with $v_i'=\pi^{-1}(v_i)$ (a bijection of admissible pairs), so the counts agree.
If $k$ and $n-k$ were both even, $2\mid\gcd(k,n(-k))$ would contradict
$\gcd(k,n)=1$.
\end{proof}

%======================================================================
\section{Reduction I: characters, and dependence on a single parameter}
\label{sec:reduction1}

\begin{proposition}\label{prop:red1}
Let $v_1,v_2$ be distinct and nonzero, $v_3=v_1+v_2$, and $\rho=v_2/v_1$
(so $\rho\notin\{0,1\}$, and $\sigma:=1+\rho=v_3/v_1$). Then
\begin{equation}\label{eq:red1}
N(v_1,v_2)\;=\;2^{2n-3}\;+\;2^{-n}\,Z(\rho),
\qquad
Z(\rho):=\sum_{\lambda\in\Fn^{\ast}}S(\lambda)\,S(\lambda\rho)\,S(\lambda\sigma),
\end{equation}
where $S$ is as in \eqref{eq:2to1}. In particular $N(v_1,v_2)=N(1,\rho)$
depends only on $\rho$, and Conjecture~\ref{conj:main} is equivalent to
\begin{equation}\label{eq:Zzero}
Z(\rho)=0\qquad\text{for all } \rho\in\Fn\setminus\{0,1\}.
\end{equation}
Moreover $N(1,\rho)$ is invariant under the $S_3$-action generated by
$\rho\mapsto1/\rho$ and $\rho\mapsto1+\rho$.
\end{proposition}

\begin{proof}
By orthogonality of additive characters,
$2^{-n}\sum_{\lambda\in\Fn}\ch(\lambda w)=[w=0]$, so
\[
N=\sum_{x,y,z\in\Delta}2^{-n}\sum_{\lambda\in\Fn}
\ch\bigl(\lambda(v_1x+v_2y+v_3z)\bigr)
=2^{-n}\sum_{\lambda\in\Fn}S(\lambda v_1)S(\lambda v_2)S(\lambda v_3).
\]
The term $\lambda=0$ contributes $|\Delta|^3 2^{-n}=2^{3(n-1)-n}=2^{2n-3}$;
reindexing $\lambda\mapsto\lambda/v_1$ in the remaining sum gives
\eqref{eq:red1}. The equation $v_1x+v_2y+v_3z=0$ with $v_1+v_2+v_3=0$ is
symmetric under simultaneously permuting $(v_1,v_2,v_3)$ and $(x,y,z)$, which
induces the stated $S_3$-action on $\rho$.
\end{proof}

\begin{remark}[the conjecture pins $N$ at its average]\label{rem:average}
Fix $v_1=1$. For $y\ne z$ the map $\rho\mapsto x=\rho y+(1+\rho)z$
is a bijection $\Fn\to\Fn$ sending $0\mapsto z$ and $1\mapsto y$; for $y=z$ it
is constant $=z$. Counting triples $(x,y,z)\in\Delta^3$ with
$x=\rho y+(1+\rho)z$ over all $\rho\notin\{0,1\}$ therefore gives
\[
\sum_{\rho\ne0,1}N(1,\rho)
=(2^n-2)\,2^{n-1}\;+\;\bigl(2^{2n-2}-2^{n-1}\bigr)\bigl(2^{n-1}-2\bigr),
\]
and dividing by $2^n-2$ yields exactly $2^{2n-3}$. Thus
$\operatorname{avg}_{\rho}N(1,\rho)=2^{2n-3}$ unconditionally, and
Conjecture~\ref{conj:main} states that $N(1,\cdot)$ is \emph{constant}.
\end{remark}

%======================================================================
\section{Reduction II: a balancedness statement for the derivative}
\label{sec:reduction2}

Write $D(b)=F(b)+F(b+1)=\delta(b)+1$ for the (normalized) derivative of the
Kasami function in direction $1$.

\begin{proposition}\label{prop:red2}
For $\rho\notin\{0,1\}$, $\sigma=1+\rho$,
\begin{equation}\label{eq:red2}
8\,Z(\rho)\;=\;2^{n}\,C(\rho)-2^{3n},\qquad
C(\rho):=\bigl|\{(b,c,e)\in\Fn^3:\ D(b)+\rho D(c)+\sigma D(e)=0\}\bigr|.
\end{equation}
Hence Conjecture~\ref{conj:main} is equivalent to:
\emph{$C(\rho)=2^{2n}$ for all $\rho\notin\{0,1\}$}, i.e.\ the map
$(b,c,e)\mapsto D(b)+\rho D(c)+\sigma D(e)$ attains $0$ exactly as often as a
balanced map would.
\end{proposition}

\begin{proof}
By \eqref{eq:2to1}, $2S(a)=\sum_b\ch(a\delta(b))$, so
\[
8Z(\rho)=\sum_{\lambda\ne0}\ \sum_{b,c,e}
\ch\Bigl(\lambda\bigl[\delta(b)+\rho\,\delta(c)+\sigma\,\delta(e)\bigr]\Bigr)
=\sum_{b,c,e}\Bigl(2^n\bigl[\delta(b)+\rho\delta(c)+\sigma\delta(e)=0\bigr]-1\Bigr).
\]
Finally $\delta=D+1$ and $1+\rho+\sigma=0$ give
$\delta(b)+\rho\delta(c)+\sigma\delta(e)=D(b)+\rho D(c)+\sigma D(e)$.
\end{proof}

%======================================================================
\section{Reduction III: an exact statement about $\MCM^{-1}$}
\label{sec:reduction3}

By Lemma~\ref{lem:frobenius} we may assume $k$ odd; then
$\MCM$ is a permutation of $\Fn$ (Lemma~\ref{lem:mcmparity}) --- for every
parity of $n$ --- and we write $\psi=\MCM^{-1}$.

\begin{theorem}\label{thm:red3}
Let $\gcd(k,n)=1$ with $k$ odd, and let $\psi=\MCM^{-1}$. For
$\rho\notin\{0,1\}$, $\sigma=1+\rho$, define
\begin{equation}\label{eq:A111}
A(\rho)\;:=\!\!\sum_{\substack{(t_1,t_2,t_3)\in\Fn^3\\ t_1+\rho t_2+\sigma t_3=0}}\!\!
\ch\bigl(\psi(t_1)+\psi(t_2)+\psi(t_3)\bigr)
\;=\;\sum_{t,a\in\Fn}\ch\bigl(\psi(t+\rho a)+\psi(t+a)+\psi(t)\bigr).
\end{equation}
Then $8\,Z(\rho)=2^{n}A(\rho)$, and Conjecture~\ref{conj:main} is equivalent to
\begin{equation}\label{eq:A111zero}
A(\rho)=0\qquad\text{for all }\rho\in\Fn\setminus\{0,1\}.
\end{equation}
Equivalently, in Fourier form, with
$W_\psi(1,\mu)=\sum_x\ch(\psi(x)+\mu x)=W_{\MCM}(\mu,1)=2S(\mu)$:
\[
A(\rho)=2^{-n}\sum_{\mu\in\Fn}W_\psi(1,\mu)\,W_\psi(1,\mu\rho)\,W_\psi(1,\mu\sigma).
\]
\end{theorem}

\begin{proof}
First, the two expressions in \eqref{eq:A111} agree: the plane
$P_\rho=\{t_1+\rho t_2+\sigma t_3=0\}$ contains the diagonal direction
$(1,1,1)$ (as $1+\rho+\sigma=0$), and
$(t,a)\mapsto(t+\rho a,\,t+a,\,t)$ is a bijection of $\Fn^2$ onto $P_\rho$:
given $(t_2,t_3)\in\Fn^2$ take $t=t_3$, $a=t_2+t_3$, and then
$\rho t_2+\sigma t_3=(\rho+\sigma)t+\rho a=t+\rho a$.

Next, by Lemma~\ref{lem:2to1} and \eqref{eq:id3},
each $b$ corresponds to $s=b^2+b\in H_0$ ($2$-to-$1$) with
$\delta(b)=\MCM(s)$, so by Proposition~\ref{prop:red2},
\[
C(\rho)=8\,C_H(\rho),\qquad
C_H(\rho):=\bigl|\{(s_1,s_2,s_3)\in H_0^3:
\MCM(s_1)+\rho\,\MCM(s_2)+\sigma\,\MCM(s_3)=0\}\bigr|.
\]
Using $\mathbf{1}_{H_0}(s)=\tfrac12(1+\ch(s))$ and substituting
$t_i=\MCM(s_i)$ (a bijection of $\Fn$, $k$ odd),
\[
8\,C_H(\rho)=\sum_{\varepsilon\in\{0,1\}^3}A_\varepsilon(\rho),
\qquad
A_\varepsilon(\rho)=\!\!\sum_{\substack{t_1+\rho t_2+\sigma t_3=0}}\!\!
\ch\bigl(\varepsilon_1\psi(t_1)+\varepsilon_2\psi(t_2)+\varepsilon_3\psi(t_3)\bigr).
\]
We claim $A_\varepsilon=0$ for all $\varepsilon\notin\{(0,0,0),(1,1,1)\}$.
Parametrize $P_\rho$ by free $(t_2,t_3)$, $t_1=\rho t_2+\sigma t_3$, and use
that $\sum_{u\in\Fn}\ch(\psi(u))=\sum_{s\in\Fn}\ch(s)=0$ ($\psi$ is a
permutation):
for $\varepsilon=(1,0,0)$, for each fixed $t_3$ the map $t_2\mapsto\rho
t_2+\sigma t_3$ is a bijection ($\rho\neq0$), so the sum is
$2^n\sum_u\ch(\psi(u))=0$; the cases $(0,1,0),(0,0,1)$ are immediate;
for $\varepsilon=(1,1,0)$, fix $t_2$ and sum over $t_3$ (bijection since
$\sigma\ne0$), factoring as
$\bigl(\sum_{t_2}\ch(\psi(t_2))\cdot 0\bigr)$-type products; similarly for
$(1,0,1)$; and $(0,1,1)$ factors as
$\bigl(\sum\ch\psi\bigr)^2=0$. Hence
$8C_H=2^{2n}+A_{(1,1,1)}=2^{2n}+A(\rho)$, i.e.\
$C(\rho)=2^{2n}+A(\rho)$, and \eqref{eq:red2} gives $8Z=2^nA$.
The Fourier form follows by detecting the plane with characters:
$\mathbf 1_{P_\rho}=2^{-n}\sum_\mu\ch(\mu(t_1+\rho t_2+\sigma t_3))$ and
$\sum_{x}\ch(\psi(x)+\mu x)=\sum_{y}\ch(y+\mu\,\MCM(y))=W_{\MCM}(\mu,1)=2S(\mu)$,
the last equality by splitting $\Fn=H_0\sqcup H_1$ as in
Section~\ref{sec:closedform} (or directly from
$2S(\mu)=\sum_b\ch(\mu\delta(b))$ and \eqref{eq:id3}).
\end{proof}

\begin{remark}
Equation \eqref{eq:A111zero} is the sharpest ``clean'' form of the conjecture we
know: a single exponential-sum identity for the fixed permutation
$\psi=\MCM^{-1}$, uniform over the $2^n-2$ planes through the diagonal.
Note the planes $P_\rho$ are exactly the $2$-dimensional subspaces of $\Fn^3$
that contain the diagonal line $\{(t,t,t)\}$ and are defined by a linear form
with all three coefficients nonzero.
\end{remark}

%======================================================================
\section{The case $k\equiv\pm1\pmod n$: full proof}\label{sec:k1}

\begin{theorem}\label{thm:k1}
Let $k\equiv\pm1\pmod n$ (and $\gcd(k,n)=1$, e.g.\ $k\in\{1,n-1\}$). Then
$\Delta=H_0$ and Conjecture~\ref{conj:main} holds, for every $n\ge2$ (both
parities).
\end{theorem}

\begin{proof}
For $k=1$, $d=3$ and
$\delta(b)=b^3+(b+1)^3+1=b^2+b$, so $\Delta=\{b^2+b\}=H_0$
(Lemma~\ref{lem:2to1}). For $k\equiv-1$, i.e.\ $k=n-1$,
Lemma~\ref{lem:frobenius} gives $\Delta_{n-1}=\Delta_1^{2^{2(n-1)}}=H_0$
since $H_0$ is invariant under field automorphisms.

Now compute $S$ for $\Delta=H_0$: for $a\in\{0,1\}$, $\Tr(at)=0$ on $H_0$, so
$S(a)=2^{n-1}$; for $a\notin\{0,1\}$ the linear form $t\mapsto\Tr(at)$ is
nonzero on the hyperplane $H_0$ (its annihilator is $\F_2=\{0,1\}$), hence
balanced, and $S(a)=0$. So $\supp S=\{0,1\}$.
In $Z(\rho)=\sum_{\lambda\neq0}S(\lambda)S(\lambda\rho)S(\lambda\sigma)$
a nonzero term requires $\lambda=\lambda\rho=1$, forcing $\rho=1$, excluded.
Hence $Z\equiv0$ and Proposition~\ref{prop:red1} concludes.
\end{proof}

\begin{remark}
Two other proofs are instructive.
(i) Linear algebra: for $\Delta=H_0$ the solution set of
$v_1x+v_2y+v_3z=0$ in $H_0^3$ is the kernel of an $\F_2$-linear map
$H_0^3\to\Fn$ of full rank $n$ (surjectivity is where $v_1\ne v_2$, $v_i\ne0$
enter), so the count is $2^{3(n-1)-n}$.
(ii) Via Theorem~\ref{thm:red3} and the Monomial Lemma below: for $k=1$,
$\MCM(x)=x$, $\psi=x^{e}$ with $e=1$, and
$P_\rho(\tau)=\tau+(\tau+1)+(\tau+\rho)=\tau+1+\rho$ has exactly one root.
\end{remark}

%======================================================================
\section{The refuted mechanism: termwise vanishing}\label{sec:termwise}

The proof above works because $\supp S\setminus\{0\}$ contains no three
distinct elements summing to zero (indeed it is a single point). This suggests:

\begin{quote}
\emph{Candidate mechanism.} For all distinct nonzero $a_1,a_2$:
$S(a_1)S(a_2)S(a_1+a_2)=0$. (This would kill every term of $Z(\rho)$
separately: the triple $\{\lambda,\lambda\rho,\lambda\sigma\}$ consists of
distinct nonzero elements summing to $0$.)
For instance, $\supp S\setminus\{0\}\subseteq H_1$ would suffice, since three
elements of $H_1$ cannot sum to $0$.
\end{quote}

\textbf{This mechanism is false in general.} Our computations
(Section~\ref{sec:numerics}, Table~\ref{tab:spectra}) show that for every
tested pair $(n,k)$ with $2\le k\le n-2$ the support of $S$ contains many
zero-sum triples of distinct elements --- e.g.\ for $(n,k)=(5,2)$ the support
has $16$ elements and there are already zero-sum triples among them; for
$(n,k)=(13,3)$ there are $4{,}193{,}280$ violating pairs --- and yet
$Z(\rho)=0$ for every $\rho$ in every tested case.
Thus \emph{any} proof of the conjecture must exploit genuine cancellation
between nonzero terms of $Z(\rho)$, not disjointness of supports.
This is the single most important structural lesson from our experiments.

For $k\equiv\pm1\pmod n$ the mechanism does hold (trivially), which is why
that case is easy.

%======================================================================
\section{The monomial mechanism: exact root counts}\label{sec:monomial}

To understand what makes \eqref{eq:A111zero} true, we isolate a family of maps
for which the analogous statement can be proved outright. For an integer $e$
with $\gcd(e,2^n-1)=1$ let $x^{e}$ denote the associated power permutation of
$\Fn$ ($0\mapsto0$; the exponent is read modulo $2^n-1$ on $\Fn^\ast$, so
$e=-1$ is the inversion map).

\begin{lemma}[Monomial Lemma]\label{lem:monomial}
Let $\psi(x)=x^{e}$ with $\gcd(e,2^n-1)=1$, and $\rho\notin\{0,1\}$. Then, with
$A(\rho)$ as in \eqref{eq:A111} (for this $\psi$),
\[
A(\rho)\;=\;2^{n}\,\bigl(r_e(\rho)-1\bigr),
\qquad
r_e(\rho):=\bigl|\{\tau\in\Fn:\ \tau^{e}+(\tau+1)^{e}+(\tau+\rho)^{e}=0\}\bigr|.
\]
The same holds with $\ch$ replaced by $\ch(w\,\cdot)$ for any $w\ne0$ (the
statement is independent of the output mask). Consequently
$A\equiv0$ on $\Fn\setminus\{0,1\}$ if and only if the trinomial-type equation
has \emph{exactly one} root for every $\rho\notin\{0,1\}$.
\end{lemma}

\begin{proof}
In $A(\rho)=\sum_{t,a}\ch(\psi(t+\rho a)+\psi(t+a)+\psi(t))$ the terms with
$a=0$ contribute $\sum_t\ch(3\psi(t))=\sum_t\ch(\psi(t))=0$.
For $a\ne0$ substitute $t=a\tau$ and use full multiplicativity of the power
map (including at $0$):
\[
\psi(a(\tau+\rho))+\psi(a(\tau+1))+\psi(a\tau)=a^{e}\,P_\rho(\tau),
\qquad P_\rho(\tau):=\tau^{e}+(\tau+1)^{e}+(\tau+\rho)^{e}.
\]
Since $a\mapsto a^{e}$ permutes $\Fn^\ast$,
\[
A(\rho)=\sum_{\tau}\sum_{u\neq0}\ch\bigl(u\,P_\rho(\tau)\bigr)
=\sum_{\tau}\Bigl(2^n\bigl[P_\rho(\tau)=0\bigr]-1\Bigr)
=2^n r_e(\rho)-2^n. \qedhere
\]
\end{proof}

\begin{theorem}[Gold permutations satisfy the identity]\label{thm:gold}
Let $n$ be odd and $e=2^j+1$ for any $j\ge1$ (then $\gcd(e,2^n-1)=1$, so
$x^{e}$ is a permutation). Then $r_e(\rho)=1$ for every $\rho\notin\{0,1\}$;
consequently $A\equiv0$ for $\psi=x^{2^j+1}$.
\end{theorem}

\begin{proof}
Write $Q=2^j$. Expanding,
\[
P_\rho(\tau)=\tau^{Q+1}+(\tau+1)^{Q+1}+(\tau+\rho)^{Q+1}
=\tau^{Q+1}+c\,\tau^{Q}+c^{Q}\tau+\bigl(1+\rho^{Q+1}\bigr),
\qquad c:=1+\rho .
\]
Substituting $\tau=z+c$ and using
$(z+c)^{Q+1}=z^{Q+1}+cz^{Q}+c^{Q}z+c^{Q+1}$:
\[
P_\rho(z+c)=z^{Q+1}+c^{Q+1}+1+\rho^{Q+1}
=z^{Q+1}+\rho+\rho^{Q},
\]
since $c^{Q+1}=(1+\rho)(1+\rho^{Q})=1+\rho+\rho^{Q}+\rho^{Q+1}$.
For $n$ odd, $\gcd(2^j+1,2^n-1)=1$ (any common divisor divides both
$2^{2j}-1$ and $2^n-1$, hence divides
$2^{\gcd(2j,n)}-1=2^{\gcd(j,n)}-1$, which divides $2^{j}-1$, coprime to $2^j+1$),
so $z\mapsto z^{Q+1}$ is a bijection and $z^{Q+1}=\rho+\rho^{Q}$ has exactly
one solution.
\end{proof}

\begin{theorem}[the inversion map satisfies the identity]\label{thm:inverse}
Let $e=-1$ (i.e.\ $\psi(x)=x^{2^n-2}$), any $n\ge2$. Then $r_e(\rho)=1$ for all
$\rho\notin\{0,1\}$; consequently $A\equiv0$ for $\psi=x^{-1}$.
\end{theorem}

\begin{proof}
First check $\tau\in\{0,1,\rho\}$ (where the convention $0^{-1}=0$ matters):
$P_\rho(0)=1+\rho^{-1}\ne0$ ($\rho\ne1$);
$P_\rho(1)=1+(1+\rho)^{-1}\ne0$ ($\rho\ne0$);
$P_\rho(\rho)=\rho^{-1}+(1+\rho)^{-1}\ne0$ ($\rho\neq1+\rho$ always).
For $\tau\notin\{0,1,\rho\}$, multiply
$\tau^{-1}+(\tau+1)^{-1}+(\tau+\rho)^{-1}=0$ by $\tau(\tau+1)(\tau+\rho)$:
\[
(\tau+1)(\tau+\rho)+\tau(\tau+\rho)+\tau(\tau+1)
=\tau^{2}+\rho=0,
\]
which has the unique solution $\tau=\rho^{1/2}$ (Frobenius is bijective);
and $\rho^{1/2}\in\{0,1,\rho\}$ would force $\rho\in\{0,1\}$.
\end{proof}

\begin{remark}[what these proofs teach us]\label{rem:mechanism}
(i) The proofs rest on an \emph{exact fibration}: the scaling action
$t=a\tau$ turns the two-dimensional sum into a one-parameter family of
complete character sums, leaving a root-counting problem, which
miraculously has the constant answer $1$ ($=$ the average).
(ii) The property is quite rigid: by Lemma~\ref{lem:monomial} it is
insensitive to the output mask for monomials, but for non-monomial $\psi$ it
is mask-sensitive: numerically, $\lambda\cdot\MCM^{-1}$ fails
\eqref{eq:A111zero} for generic $\lambda\neq1$, while $\MCM^{-1}(ax)+c$
satisfies it (Section~\ref{sec:numerics}). This matches the general fact that
inner scalings commute with the $\rho$-dilation structure while outer scalings
change the character.
(iii) Numerically, \emph{random} permutations fail \eqref{eq:A111zero}
decisively; the identity is a genuine structural property.
(iv) $\MCM$ itself (as opposed to $\MCM^{-1}$) \emph{fails}
\eqref{eq:A111zero} for $2\le k\le n-2$: the property is really about the
inverse map. Since $\MCM^{-1}$ is not a monomial for $2\le k\le n-2$, the
scaling fibration is not available; this was the principal obstacle to a
proof of the general case, circumvented for $k=2$ in Section~\ref{sec:k2}
by passing to a quadratic model and a different monomial substitution.
(v) For the Kasami exponent itself, $r_{d(k)}(\rho)$ is \emph{not} identically
$1$ (e.g.\ $(n,k)=(7,2)$ gives root counts $\{0,1,2\}$), so the naive
transplantation of the monomial mechanism to $F$ fails as well
(consistent with (iv)).
\end{remark}

%======================================================================
\section{A closed form for $S$, and a second parametrization of $\Delta$}
\label{sec:closedform}

Throughout this section $n$ is odd (no assumption on the parity of $k$),
$q=2^k$, $\gcd(k,n)=1$. Then $\gcd(q+1,2^n-1)=1$ (as in the proof of
Theorem~\ref{thm:gold}), so $x\mapsto x^{q+1}$ is a permutation of $\Fn$, and
exponents $\tfrac1{q+1}$, $q^2-q+1=\tfrac{q^3+1}{q+1}$ act on $\Fn^\ast$ in the
usual way. Recall $(q+1)\,d=q^3+1$, so that
\begin{equation}\label{eq:goldsub}
F\bigl(x^{q+1}\bigr)=x^{q^{3}+1}\qquad\text{(the Gold substitution).}
\end{equation}

\begin{theorem}[closed form]\label{thm:closedform}
Let $n$ be odd, $\gcd(k,n)=1$, and define the rational map
\[
\Phi(m)\;=\;\frac{m+m^{q}+m^{q^{2}}}{m^{\,q^{2}-q+1}}
\qquad(m\in\Fn^{\ast}).
\]
Then:
\begin{enumerate}
\item[(i)] The map $b\mapsto m(b):=\bigl(x+y\bigr)^{-(q+1)}$, where
$x=b^{1/(q+1)}$, $y=(b+1)^{1/(q+1)}$, is exactly $2$-to-$1$ from $\Fn$ onto
$H_1$, with fibers $\{b,b+1\}$, and
\[
D(b)\;=\;F(b)+F(b+1)\;=\;\Phi\bigl(m(b)\bigr).
\]
\item[(ii)] $\Phi$ restricted to $H_1$ is a bijection onto $\Delta+1=\{t+1:
t\in\Delta\}$.
\item[(iii)] For every $a\in\Fn$,
\begin{equation}\label{eq:closedform}
S(a)\;=\;\ch(a)\sum_{m\in H_1}\ch\bigl(a\,\Phi(m)\bigr).
\end{equation}
\item[(iv)] Conjecture~\ref{conj:main} is equivalent to: for all
$\rho\notin\{0,1\}$, $\sigma=1+\rho$,
\[
\bigl|\{(m_1,m_2,m_3)\in H_1^3:\ \Phi(m_1)+\rho\,\Phi(m_2)+\sigma\,\Phi(m_3)=0\}\bigr|
\;=\;2^{2n-3}.
\]
\end{enumerate}
\end{theorem}

\begin{proof}
\emph{(i)}
Fix $b$ and let $x=b^{1/(q+1)}$, $y=(b+1)^{1/(q+1)}$, so that by
\eqref{eq:goldsub}, $F(b)=x^{q^3+1}$, $F(b+1)=y^{q^3+1}$, and
$x^{q+1}+y^{q+1}=1$; moreover $b\leftrightarrow(x,y)$ is a bijection onto
$\{(x,y):x^{q+1}+y^{q+1}=1\}$, and swapping $b\leftrightarrow b+1$ swaps
$x\leftrightarrow y$. Put $e=x+y$; then $e\neq0$ (else $0=1$).
Writing $y=x+e$ and expanding,
\begin{equation}\label{eq:xeq}
x^{q+1}+(x+e)^{q+1}=e\,x^{q}+e^{q}x+e^{q+1}=1
\quad\Longleftrightarrow\quad
z^{q}+z=1+e^{-(q+1)}=:c_e,\qquad z:=x/e .
\end{equation}
The additive map $\lambda(z)=z^{q}+z$ has kernel
$\F_{2^k}\cap\Fn=\F_2$ (as $\gcd(k,n)=1$) and image $H_0$
(the image lies in $H_0$ since $\Tr(z^q)=\Tr(z)$, and has size $2^{n-1}$).
So for a given $e\ne0$, \eqref{eq:xeq} has solutions ($2$ of them, $\{x,x+e\}$)
iff $\Tr(c_e)=0$, i.e.\ iff $\Tr\bigl(e^{-(q+1)}\bigr)=\Tr(1)=1$ ($n$ odd).
Since each of the $2^n$ elements $b$ produces exactly one pair $(x,x+e)$ and
each admissible $e$ carries exactly two such pairs (namely $(x,x{+}e)$ and
$(x{+}e,x)$, corresponding to $b$ and $b{+}1$), the map $b\mapsto e(b)$ is
exactly $2$-to-$1$ onto $\{e:\Tr(e^{-(q+1)})=1\}$, with fibers $\{b,b+1\}$.
As $n$ is odd, $m:=e^{-(q+1)}$ is a bijective change of variable
$\{e\}\to\Fn^\ast$, and the admissibility condition becomes $\Tr(m)=1$, i.e.\
$m\in H_1$ (note $0\notin H_1$). Thus $b\mapsto m(b)$ is exactly $2$-to-$1$
onto $H_1$ with fibers $\{b,b+1\}$.

Next we compute $D(b)$ in terms of $m$. Using \eqref{eq:goldsub},
\[
D(b)=x^{q^{3}+1}+(x+e)^{q^{3}+1}
=e\,x^{q^{3}}+e^{q^{3}}x+e^{q^{3}+1}
=e^{q^{3}+1}\bigl(z^{q^{3}}+z+1\bigr).
\]
Let $R(w):=w+w^{q}+w^{q^{2}}$. In the commutative ring
$\F_2[\sigma_k]$ of $\F_2$-polynomials in the Frobenius $\sigma_k(w)=w^q$,
we have $z^{q^3}+z=\lambda(R(z))=R(\lambda(z))$, since
$(T+1)(T^2+T+1)=T^3+1$. Hence, by \eqref{eq:xeq},
\[
z^{q^{3}}+z=R(c_e)=R(1)+R\bigl(e^{-(q+1)}\bigr)=1+R(m),
\]
because $R(1)=1+1+1=1$. Therefore
\[
D(b)=e^{q^{3}+1}\,R(m)
= m^{-\frac{q^{3}+1}{q+1}}R(m)
= \frac{R(m)}{m^{\,q^{2}-q+1}}
= \Phi(m).
\]

\emph{(ii)} By (i), the image of $D$ is $\Phi(H_1)$; but the image of
$D=\delta+1$ is $\Delta+1$, of size $2^{n-1}=|H_1|$; hence $\Phi|_{H_1}$ is a
bijection onto $\Delta+1$.

\emph{(iii)} Using \eqref{eq:2to1} and (i),
\[
2S(a)=\sum_b\ch\bigl(a\delta(b)\bigr)=\ch(a)\sum_b\ch\bigl(a D(b)\bigr)
=\ch(a)\cdot2\sum_{m\in H_1}\ch\bigl(a\Phi(m)\bigr).
\]

\emph{(iv)} Insert \eqref{eq:closedform} into
$Z(\rho)=\sum_{\lambda\ne0}S(\lambda)S(\lambda\rho)S(\lambda\sigma)$ and note
$\ch(\lambda)\ch(\lambda\rho)\ch(\lambda\sigma)=\ch(\lambda(1+\rho+\sigma))=1$:
\[
Z(\rho)=\sum_{m_1,m_2,m_3\in H_1}\ \sum_{\lambda\neq0}
\ch\Bigl(\lambda\bigl[\Phi(m_1)+\rho\Phi(m_2)+\sigma\Phi(m_3)\bigr]\Bigr)
=2^n\,\#\{\cdots\}-2^{3(n-1)},
\]
and $Z\equiv0$ iff the count equals $2^{2n-3}$
(Proposition~\ref{prop:red1}).
\end{proof}

\begin{remark}
(a) For $k=1$, $\Phi(m)=(m+m^2+m^4)/m^3=m+m^{-1}+m^{-2}$ and
$\Tr(\Phi(m))=\Tr(m)$, so $\Phi(H_1)\subseteq H_1=H_0+1$, consistent with
$\Delta=H_0$ (Theorem~\ref{thm:k1}).
(b) Equation \eqref{eq:closedform} was verified numerically for every
admissible odd-$n$ pair with $n\le13$ (all $a$ for $n\le11$; sampled $a$ for
$n=13$); see Section~\ref{sec:numerics}.
(c) The two parametrizations $\Delta=\MCM(H_0)$ and $\Delta+1=\Phi(H_1)$ are
genuinely different rational models of the same set; the second has the
advantage that the trace condition sits on the \emph{variable} side, which is
what makes the manipulation in (iv) character-free.
\end{remark}

%======================================================================
\section{The case $k=2$ ($d=13$): a proof of the conjecture}\label{sec:k2}

In this section we \emph{prove} Conjecture~\ref{conj:main} in the first
genuinely open case, $k=2$, i.e.\ $d=13$, $q=4$; note $\gcd(2,n)=1$ forces
$n$ \emph{odd} throughout. The proof has three stages:
(1) a quadratic model turning all character sums into quadratic Gauss sums,
governed by a \emph{master correspondence} between the radicals of the pencil
and two low-degree algebraic families;
(2) a counting theorem ($G_1\equiv0$, Theorem~\ref{claim:Cprime}) which
determines the spectrum completely and shows that $\MCM=x(x+1)^5$ is exactly
$2$-to-$1$;
(3) a monomial change of variables converting the entire triple correlation
into an exact root count for an explicit polynomial family, which is then
pinned down by an explicit root together with the unconditional average
identity of Remark~\ref{rem:average}.
No Arf-invariant or sign computations are needed anywhere.

\subsection{The quadratic model}

For $k=2$,
\[
\MCM(x)=\frac{T_2(x)^{5}}{x^{4}}=\frac{(x+x^2)^5}{x^4}=x(x+1)^5=:g(x)
=x^{6}+x^{5}+x^{2}+x ,
\]
a sextic polynomial all of whose exponents have binary weight $\le2$.
Consequently, for $u\ne0$, using $\Tr(y^2)=\Tr(y)$,
\begin{equation}\label{eq:quadmodel}
\Tr\bigl(u\,g(s)\bigr)=\Tr\bigl(us^{5}+\sqrt{u}\,s^{3}+(u+\sqrt{u})s\bigr),
\end{equation}
a \emph{quadratic} Boolean form in $s$ (the exponents $5=2^2{+}1$ and
$3=2{+}1$ are Gold exponents). Define for $\varepsilon\in\{0,1\}$
\[
G_\varepsilon(u):=\sum_{s\in\Fn}\ch\bigl(u\,g(s)+\varepsilon s\bigr),
\qquad\text{so}\qquad
2S(u)=G_0(u)+G_1(u)=\sum_b\ch\bigl(u\,\delta(b)\bigr)=:H(u)
\]
by Lemma~\ref{lem:2to1} and $\Delta=g(H_0)$.

\begin{lemma}[radical of the pencil]\label{lem:radical}
Let $u\neq0$ and $Q_u(s)=\Tr(us^5+\sqrt u\,s^3)$. The radical
$V_u=\{w: Q_u(s+w)+Q_u(s)+Q_u(w)=0\ \forall s\}$ is the kernel of the
linearized polynomial
\[
A_u(s)=u^{4}s^{16}+u^{2}s^{8}+u s^{2}+u s ,
\]
$\dim_{\F_2}V_u=:v_u\in\{1,3\}$, and $G_\varepsilon(u)\in\{0,\pm2^{(n+v_u)/2}\}$
with
\begin{equation}\label{eq:compat}
G_\varepsilon(u)\ne0\iff
\Tr\bigl(u\,g(w)\bigr)+\varepsilon\,\Tr(w)=0\quad\text{for all }w\in V_u .
\end{equation}
\end{lemma}

\begin{proof}
Polarizing \eqref{eq:quadmodel}: the associated symplectic form is
$B_u(s,w)=\Tr\bigl(u(sw^4+s^4w)+\sqrt u(s^2w+sw^2)\bigr)
=\Tr\bigl(s\cdot L_u(w)\bigr)$ with
$L_u(w)=uw^4+\sqrt u\,w^2+(uw)^{1/4}+u^{1/4}w^{1/2}$, and $A_u=L_u^{4}$,
so $\ker L_u=\ker A_u$. Since $\deg A_u=2^4$, $v_u\le4$; since the rank
$n-v_u$ of a symplectic form is even and $n$ is odd, $v_u\in\{1,3\}$.
The value statement is the standard theory of quadratic Boolean forms:
the sum vanishes unless the affine part is compatible on the radical, and
then has absolute value $2^{(n+v)/2}$. For the compatibility criterion,
observe that for $w\in V_u$ and all $s$,
$g(s+w)+g(s)+g(w)$ is the field-valued polar form of the quadratic part of
$g$ plus linear terms, and taking $\Tr(u\,\cdot)$ of it gives
$B_u(s,w)=0$; hence
$\ch(u g(s{+}w)+\varepsilon(s{+}w))
=\ch(ug(s)+\varepsilon s)\,\ch(ug(w)+\varepsilon w)$,
so $G_\varepsilon(u)\ne0$ forces $\Tr(ug(w))+\varepsilon\Tr(w)=0$ on $V_u$;
conversely if the affine form vanishes on $V_u$ the standard
rank-$(n-v)$ evaluation applies.
\end{proof}

We write $T(u,w):=\Tr\bigl(w+u\,g(w)\bigr)$ for $w\in V_u$; thus
$G_1(u)\neq0$ iff $T(u,\cdot)\equiv0$ on $V_u$, and
$G_0(u)\ne0$ iff $\Tr(u\,g(\cdot))\equiv0$ on $V_u$.

\subsection{The fibers of $g$}

\begin{lemma}[trace separation]\label{lem:traceodd}
Let $X\neq Y$ with $g(X)=g(Y)$. Then $\Tr(X+Y)=1$. Consequently every fiber
of $g$ has at most two elements, contains at most one element of $H_0$ and at
most one of $H_1$; in particular $g$ is injective on $H_0$ and on $H_1$
separately.
\end{lemma}

\begin{proof}
Let $e_1=X+Y\ne0$, $e_2=XY$. Using power sums
($p_5=e_1^5+e_1^3e_2+e_1e_2^2$, $p_6=e_1^6+e_1^2e_2^2$ in characteristic $2$),
\[
g(X)+g(Y)=e_1\bigl[(e_1{+}1)e_2^{2}+e_1^{2}e_2+(e_1{+}1)^{5}\bigr]\Big/e_1
\quad\text{vanishes iff}\quad
(e_1{+}1)e_2^{2}+e_1^{2}e_2+(e_1{+}1)^{5}=0 .
\]
If $e_1=1$ this forces $e_2=0$, i.e.\ $\{X,Y\}=\{0,1\}$, and
$\Tr(1)=1$ ($n$ odd). If $e_1\neq1$, divide by $(e_1+1)$ and substitute
$W=(e_1+1)e_2/e_1^{2}$ to get $W^{2}+W=(e_1+1)^{6}/e_1^{4}$. Since
$(e_1+1)^6/e_1^4=e_1^{2}+1+e_1^{-2}+e_1^{-4}$ and
$\Tr(e_1^{-2})=\Tr(e_1^{-4})=\Tr(e_1^{-1})$, solvability requires
$0=\Tr(e_1^2)+\Tr(1)=\Tr(e_1)+1$, i.e.\ $\Tr(e_1)=1$. As $e_2=XY$ furnishes a
solution, $\Tr(X+Y)=1$.
For a three-element fiber $\{X,Y,Z\}$ the pairwise sums add to $0$, but their
traces would add to $1+1+1=1\ne\Tr(0)$ --- impossible.
\end{proof}

\begin{proposition}[exactness criteria]\label{prop:claimD}
The following are equivalent:
(i) $g$ is exactly $2$-to-$1$ on $\Fn$ (each fiber one element of $H_0$, one
of $H_1$); (ii) $g(H_1)=g(H_0)=\Delta$; (iii) $G_1(u)=0$ for \emph{all} $u$;
(iv) every radical $V_u$ ($u\ne0$) contains an element $w$ with $T(u,w)=1$.
All four statements are proved in Theorem~\ref{claim:Cprime} below.
\end{proposition}

\begin{proof}[Proof of the equivalences]
(i)$\Leftrightarrow$(ii): by Lemma~\ref{lem:traceodd}, fibers have $\le2$
elements, one from each half; exactness says no singleton fibers, i.e.\ every
value of $g|_{H_1}$ is a value of $g|_{H_0}$ and conversely; both restrictions
are injective and $|H_0|=|H_1|$.
(ii)$\Leftrightarrow$(iii): $\sum_{t\in g(H_i)}\ch(ut)$ equals
$\tfrac12(G_0+G_1)$ for $i=0$ and $\tfrac12(G_0-G_1)$ for $i=1$; two subsets
of equal size coincide iff their Fourier transforms agree, iff $G_1\equiv0$.
(iii)$\Leftrightarrow$(iv): Lemma~\ref{lem:radical}: $G_1(u)=0$ iff
$T(u,\cdot)\not\equiv0$ on $V_u$; on the $\F_2$-space $V_u$ the function
$T(u,\cdot)$ is linear (proof of Lemma~\ref{lem:radical}), so
$\not\equiv0$ iff some $w\in V_u$ has $T=1$.
\end{proof}

\subsection{The master correspondence}

The incidence $w\in V_u\setminus\{0\}$, i.e.\ $u^3w^{16}+uw^{8}+w^2+w=0$
(divide $A_u(w)=0$ by $uw$ and multiply back), turns out to be governed by two
classical low-degree families.

\begin{theorem}[master correspondence]\label{thm:master}
Let $u\neq0$, $w\neq0$, and put $z:=uw^{4}$, $\zeta:=1+1/w$,
$\nu(x):=x^{3}+x$. Then:
\begin{enumerate}
\item[(a)] $w\in V_u\iff\nu(z)=\nu(\zeta)\iff
(z+\zeta)\bigl(z^{2}+\zeta z+\zeta^{2}+1\bigr)=0$.
\item[(b)] \emph{(Type I: $z=\zeta$.)} Equivalently $uw^{5}=w+1$; then
\[
w=\frac1{z+1},\qquad u=z(z+1)^4=z^{5}+z .
\]
\item[(c)] \emph{(Type II: $z^2+\zeta z+\zeta^2+1=0$.)} Equivalently, with
$\tau:=\bigl(u/z^{5}\bigr)^{1/4}$:
\[
\tau^{2}+\tau=1+\frac1z,\qquad u=z^{5}\tau^{4},\qquad w=\frac1{z\tau}.
\]
\item[(d)] For fixed $w\notin\{0,1\}$, the incident $u$'s number
$1+2\cdot[\Tr(1/(w{+}1))=0]$; for $w=1$ the only incident $u$ is $u=1$.
\item[(e)] \emph{(Counting.)} Double counting incidences yields
\[
\#U_3:=\#\{u\ne0: v_u=3\}=\frac{2^{\,n-1}-1}{3},\qquad
\#\{u: v_u=1\}=\frac{5\cdot2^{\,n-1}-2}{3}.
\]
\end{enumerate}
\end{theorem}

\begin{proof}
(a) For $w\ne0$, dividing $A_u(w)=0$ by $uw$ gives the incidence relation
$u^{3}w^{15}+uw^{7}+w+1=0$ (after multiplying back by $w$).
On the other hand, multiplying $\nu(z)+\nu(\zeta)=0$ by $w^{3}$ and using
$w\zeta=w+1$, $z=uw^{4}$:
\[
w^{3}z^{3}+w^{3}z+(w\zeta)^{3}+w^{2}\cdot w\zeta
=u^{3}w^{15}+uw^{7}+(w+1)^{3}+w^{2}(w+1)
=u^{3}w^{15}+uw^{7}+w+1,
\]
since $(w+1)^{3}+w^{2}(w+1)=(w+1)\bigl[(w+1)^{2}+w^{2}\bigr]=w+1$.
So the two relations agree. The factorization of $\nu(z)+\nu(\zeta)$ is
standard: $x^{3}+x+c$ with root $\zeta$ factors as
$(x+\zeta)(x^{2}+\zeta x+\zeta^{2}+1)$.
(b) $z=\zeta$ means $uw^{4}=1+1/w$, i.e.\ $uw^{5}=w+1$; then $w(z+1)=1$ and
$u=z/w^{4}=z(z+1)^{4}=z^{5}+z$ since $(z+1)^4=z^4+1$.
(c) Note $\zeta+1=1/w$. Setting $\tau=(u/z^5)^{1/4}$ one computes
$z\tau=z\cdot u^{1/4}z^{-5/4}=(u/z)^{1/4}=(w^{-4})^{1/4}=1/w$,
so $\zeta+1=z\tau$ and $w=1/(z\tau)$. Substituting $\zeta=1+z\tau$ into
$z^{2}+\zeta z+\zeta^{2}+1=0$ gives
$z^{2}(\tau^{2}+\tau+1)+z=0$, i.e.\ $\tau^{2}+\tau=1+1/z$ (using $z\neq0$).
Conversely these relations imply the quadratic.
(d) For $w\notin\{0,1\}$: $\zeta\notin\{0,1\}$, the type-I root $z=\zeta$
gives one nonzero $u$; the quadratic has two further roots in $\Fn$ iff
$\Tr\bigl((\zeta^{2}+1)/\zeta^{2}\bigr)=0\iff\Tr(1/\zeta)=1
\iff\Tr\bigl(1/(w{+}1)\bigr)=0$
(as $1/\zeta=w/(w{+}1)=1+1/(w{+}1)$ and $\Tr(1)=1$), and these roots are
$\neq0,\zeta$. For $w=1$: $\zeta=0$, $\nu(z)=0$ has the single nonzero root
$z=1$, giving $u=1$.
(e) Each $u\neq0$ carries $2^{v_u}-1$ incidences, so
\[
\#\{v_u{=}1\}+7\,\#U_3=\sum_{u\neq0}(2^{v_u}-1)
=1+\sum_{w\neq0,1}\Bigl(1+2\bigl[\Tr(\tfrac1{w+1})=0\bigr]\Bigr)
=1+(2^{n}-2)+2(2^{n-1}-1),
\]
because $x=1/(w+1)$ ranges bijectively over $\Fn^\ast\setminus\{1\}$ and
$\#\{x\ne0,1:\Tr(x)=0\}=2^{n-1}-1$ ($\Tr(1)=1$). With
$\#\{v_u{=}1\}=2^{n}-1-\#U_3$ this gives $6\#U_3=2^{n}-2$.
\end{proof}

\subsection{The compatibility values}

\begin{theorem}[$T$-values]\label{thm:Tvalues}
With the notation of Theorem~\ref{thm:master}:
\begin{enumerate}
\item[(i)] Every type-I pair has $T(u,w)=1$; moreover
$\Tr(w)=\Tr\bigl(1/(z+1)\bigr)$ and
$\Tr\bigl(ug(w)\bigr)=1+\Tr(w)$.
\item[(ii)] Every type-II pair has $T(u,w)=\Tr(\tau)$. In particular the two
type-II partners $u=z^5\tau^4$ and $u'=z^5(\tau+1)^4$ of a given $z$ have
opposite $T$-values.
\end{enumerate}
\end{theorem}

\begin{proof}
(i) Write $A=1/(z+1)=w$, so $z=(1+A)/A$. Then $ug(w)=u\,w(w{+}1)^{5}$ with
$u=z(z{+}1)^{4}$ and $w+1=z/(z+1)$, so
\[
ug(w)=z(z+1)^4\cdot\frac{1}{z+1}\cdot\frac{z^{5}}{(z+1)^{5}}
=\frac{z^{6}}{(z+1)^{2}}=z^{6}A^{2}=\frac{(1{+}A)^{6}}{A^{4}}
=A^{-4}+A^{-2}+1+A^{2},
\]
using $(1+A)^{6}=(1+A^{2})(1+A^{4})$. Hence
\[
T=\Tr(w+ug(w))=\Tr(A)+\Tr(A^{-4})+\Tr(A^{-2})+\Tr(1)+\Tr(A^{2})
=\Tr(A)+0+1+\Tr(A)=1,
\]
using $\Tr(A^{-4})=\Tr(A^{-2})=\Tr(A^{-1})$ and $\Tr(A^2)=\Tr(A)$.
The remaining statements are immediate ($w=A$; $\Tr(ug(w))=T+\Tr(w)+\cdots$
i.e.\ $\Tr(ug(w))=1+\Tr(w)$ from $T=1$).
(ii) Let $\kappa:=1+1/z=\tau^{2}+\tau$, so $\tau(\tau+1)=\kappa$ and
$w=1/(z\tau)$. First,
$ug(w)=uw(w+1)^5$ and a short computation (expanding
$(1+z\tau)^{5}=(1+z\tau)(1+z^4\tau^4)$ and reducing all powers
$\tau^{2},\tau^{3}$ via $\tau^{2}=\tau+\kappa$) gives
\[
T=\Tr\Bigl(\frac1{z\tau}+\frac1{z\tau^{2}}+\frac1{\tau}\Bigr),
\]
because the terms $z^{3}\tau^{2}+z^{4}\tau^{3}$ contribute
$\Tr\bigl(z^3\tau+z^3+z\bigr)+\Tr\bigl(z^3\tau+z+z^3\bigr)=0$.
Now $\tau(\tau+1)=\kappa$ gives $1/\tau=(\tau+1)/\kappa$ and
$1/\tau^{2}=(\tau^{2}+1)/\kappa^{2}=(\tau+\kappa+1)/\kappa^{2}$. Setting
$A:=1/(z+1)$ and using $\kappa=(z+1)/z$, these become
$1/(z\tau)=A(\tau+1)$, $1/(z\tau^{2})=(z\tau+1)/(z+1)^2$ and
$1/\tau=z(\tau+1)/(z+1)$. The first and third terms of $T$ combine to
$(\tau+1)(1+z)/(z+1)=\tau+1$, and expanding the middle term,
\[
T=\Tr\bigl((\tau+1)(A^{2}+A+1)\bigr)+\Tr(A).
\]
Finally, from $\tau^{2}+\tau=\kappa$ one has, for every $\alpha$,
$\Tr(\alpha^{1/2}\tau)+\Tr(\alpha\tau)=\Tr(\alpha\kappa)$; applying this with
$\alpha=B^{2}$, $B:=A+1=1/\kappa$, gives
$\Tr(B^{2}\tau)+\Tr(B\tau)=\Tr(B)$, whence
$\Tr\bigl((B^{2}+B+1)\tau\bigr)=\Tr(B)+\Tr(\tau)$ and, using
$A^2+A+1=B^2+B+1$ and $\Tr(B^{2}+B+1)=\Tr(1)=1$,
\[
T=\bigl(\Tr(B)+\Tr(\tau)\bigr)+1+\Tr(A)=\Tr(\tau)+\Tr(A+B)+1
=\Tr(\tau)+\Tr(1)+1=\Tr(\tau).
\]
\end{proof}

\subsection{Bluher--Helleseth--Kholosha input and the classification}

By Theorem~\ref{thm:master}(b), the type-I elements of $V_u$ are in bijection
(via $z\mapsto w=1/(z+1)$) with the roots of
\begin{equation}\label{eq:bluher}
z^{5}+z+u=0
\end{equation}
--- the Bluher family $x^{q+1}+x+a$ with $q=4$. For $n$ odd
(i.e.\ $\gcd(2,n)=1$), the number $N_5(u)$ of $\Fn$-roots of
\eqref{eq:bluher} lies in $\{0,1,3\}$, with
\[
\#\{u: N_5=3\}=\frac{2^{n-1}-1}{3},\qquad
\#\{u: N_5=1\}=2^{n-1}-1,\qquad
\#\{u: N_5=0\}=\frac{2^{n}+1}{3},
\]
by the results of Bluher and Helleseth--Kholosha
\cite{Bluher,HellesethKholosha}. (Our Theorem~\ref{thm:master}(d,e) gives an
independent derivation of these counts through the $w$-side of the
correspondence.)

\begin{theorem}[classification of radicals]\label{thm:classification}
For every $u\neq0$:
\[
v_u=3\iff N_5(u)=3 ,
\]
and then $V_u\setminus\{0\}$ consists of exactly $3$ type-I and $4$ type-II
elements. If $v_u=1$, the unique radical generator is of type I iff
$N_5(u)=1$, of type II iff $N_5(u)=0$. Consequently $G_1(u)=0$ for all
$u\in U_3$, and $H=G_0$ on $U_3$.
\end{theorem}

\begin{proof}
$z\mapsto w=(z/u)^{1/4}$ is injective, so $N_5(u)=3$ produces three distinct
nonzero elements of $V_u$, forcing $v_u=3$. Thus
$\{u:N_5(u)=3\}\subseteq U_3$; since both sets have cardinality
$(2^{n-1}-1)/3$ (Bluher--Helleseth--Kholosha resp.\
Theorem~\ref{thm:master}(e)), they are equal. The type split follows:
$N_5(u)=3$ type-I elements and $7-3=4$ type-II. For $v_u=1$ the statement is
the definition of the types. Finally, on $U_3$ the radical contains a type-I
element, which has $T=1$ (Theorem~\ref{thm:Tvalues}); by
Lemma~\ref{lem:radical}, $G_1(u)=0$.
\end{proof}

\subsection{The counting theorem: $G_1\equiv0$}

We now prove the statement that was isolated, in an earlier stage of this
project, as ``Claim $\mathrm{C}'$''. The key observation is that
$T(u,\cdot)$ is \emph{$\F_2$-linear on $V_u$}: for $w,w'\in V_u$, evaluating
the factorization in the proof of Lemma~\ref{lem:radical} at $s=w'$ gives
$(-1)^{T(u,w+w')}=(-1)^{T(u,w)}(-1)^{T(u,w')}$, and $T(u,0)=0$.
Call a type-II pair $(u,w)$ \emph{null} if $T(u,w)=\Tr(\tau)=0$.

\begin{theorem}[key counting theorem]\label{claim:Cprime}
Let $n$ be odd, $k=2$. For every $u\neq0$ with $v_u=1$ the unique radical
pair has $T(u,w_u)=1$; moreover no $u\in U_3$ has $T(u,\cdot)\equiv0$ on
$V_u$. Consequently:
\begin{enumerate}
\item[(i)] $G_1(u)=0$ for \emph{every} $u\in\Fn$;
\item[(ii)] $g(x)=x(x+1)^5=\MCM$ is exactly $2$-to-$1$ on $\Fn$, each fiber
containing one element of $H_0$ and one of $H_1$; in particular
$\MCM(H_0)=\MCM(H_1)=\Delta$;
\item[(iii)] $2S(u)=H(u)=G_0(u)=\sum_{s\in\Fn}\ch\bigl(u\,g(s)\bigr)$ for all
$u\neq0$.
\end{enumerate}
The proof uses only Theorems~\ref{thm:master} and~\ref{thm:Tvalues}; in
particular it is independent of the Bluher--Helleseth--Kholosha input used in
Theorem~\ref{thm:classification}.
\end{theorem}

\begin{proof}
\emph{Global count of null pairs.} By Theorem~\ref{thm:master}(c), type-II
pairs are indexed by $(z,\tau)$ with $\tau^2+\tau=1+1/z$, solvable iff
$\Tr(1+1/z)=0$, i.e.\ $\Tr(1/z)=1$ ($n$ odd, $\Tr(1)=1$). For $z=1$ the roots
are $\tau\in\{0,1\}$; $\tau=0$ gives $u=0$ (excluded) and $\tau=1$ gives the
pair $(1,1)$ with $T=\Tr(1)=1$: not null. For each of the $2^{n-1}-1$
admissible $z\neq1$ the two roots $\tau,\tau+1$ give two pairs (with distinct
$u$'s, since $z^5\tau^4\neq z^5(\tau+1)^4$) whose $T$-values are
$\Tr(\tau)$ and $\Tr(\tau)+1$: exactly \emph{one} null pair per $z$. Hence
the total number of null pairs is exactly $2^{n-1}-1$.

\emph{Local count.} Let $x=\#\{u: v_u=1,\ \text{its pair is null}\}$ and
$y=\#\{u\in U_3:\ T(u,\cdot)\equiv0\text{ on }V_u\}$. A $v_u=1$ element
contributes one null pair iff it is counted by $x$ (a type-I pair has $T=1$
by Theorem~\ref{thm:Tvalues}(i)). For $u\in U_3$ with $T(u,\cdot)\not\equiv0$,
linearity makes $\ker T(u,\cdot)$ an index-$2$ subgroup of $V_u\cong\F_2^3$:
exactly $3$ nonzero elements have $T=0$, and they are all of type II (type-I
elements have $T=1$): exactly $3$ null pairs. For $u\in U_3$ with
$T(u,\cdot)\equiv0$, all $7$ pairs are null. Therefore, using
$\#U_3=(2^{n-1}-1)/3$ (Theorem~\ref{thm:master}(e)),
\[
2^{n-1}-1\;=\;x+3(\#U_3-y)+7y\;=\;x+(2^{n-1}-1)+4y ,
\]
whence $x+4y=0$ and $x=y=0$. This proves the first two assertions
($T(u,w_u)=1$ on $v_u=1$: type-I pairs by Theorem~\ref{thm:Tvalues}(i),
type-II pairs because they are not null).

(i) For $v_u=1$, $T(u,\cdot)$ attains $1$, and for $u\in U_3$,
$T(u,\cdot)\not\equiv0$ ($y=0$); in both cases Lemma~\ref{lem:radical} gives
$G_1(u)=0$; finally $G_1(0)=\sum_s\ch(s)=0$.
(ii) follows from (i) via Proposition~\ref{prop:claimD}
((iii)$\Rightarrow$(i),(ii) there), and (iii) from $2S=G_0+G_1$
(Subsection~\ref{sec:k2}.1) together with (i).
\end{proof}

\subsection{Consequences: the complete description of $S$ for $k=2$}

\begin{theorem}[support and values]
\label{thm:k2support}
Let $n$ be odd and $u\ne0$. Then $2S(u)=H(u)=G_0(u)$ and:
\begin{enumerate}
\item[(i)] If $v_u=1$ with radical generator $w_u$:
\[
H(u)=\pm2^{(n+1)/2}\ \text{if}\ \Tr(w_u)=1,\qquad H(u)=0\ \text{otherwise.}
\]
In the type-I case ($N_5(u)=1$, root $z$): $\Tr(w_u)=\Tr\bigl(1/(z+1)\bigr)$.
\item[(ii)] If $u\in U_3$:
$H(u)=\pm2^{(n+3)/2}$ if $\Tr(u\,g(w))=0$ for all $w\in V_u$, else $H(u)=0$.
\item[(iii)] The spectrum of $2S$ is contained in
$\{0,\pm2^{(n+1)/2},\pm2^{(n+3)/2}\}$, in accordance with
Table~\ref{tab:spectra}.
\end{enumerate}
\end{theorem}

\begin{proof}
By Theorem~\ref{claim:Cprime}, $G_1\equiv0$, so $H=G_0$. On $v_u=1$: by
\eqref{eq:compat}, $G_0(u)\neq0$ iff $\Tr(ug(w_u))=0$; since $T(u,w_u)=1$
(Theorem~\ref{claim:Cprime}), $\Tr(ug(w_u))=1+\Tr(w_u)$, which vanishes iff
$\Tr(w_u)=1$. The value is $\pm2^{(n+1)/2}$ by Lemma~\ref{lem:radical}.
(ii) is \eqref{eq:compat} for $v=3$, and (iii) collects the cases.
\end{proof}

\subsection{Elimination of the signs: reduction to an exact root count}

At this point one could try to determine the signs of $H=G_0$ (Arf
invariants of the pencil $Q_u$) and prove the cancellation
$\sum_{\mu}H(\mu)H(\mu\rho)H(\mu\sigma)=0$ directly; our experiments show the
sign function obeys no monomial-character law, so this route is hard.
It is also unnecessary: a monomial change of variables removes the signs
altogether. Recall that for $n$ odd all of $2,3,5,6$ are invertible modulo
$2^n-1$, so fractional exponents such as $5/3:=5\cdot3^{-1}$ define power
\emph{bijections} of $\Fn$ (fixing $0$).

\begin{theorem}[monomial reduction]\label{thm:monred}
Let $n$ be odd, $\rho\notin\{0,1\}$, $\sigma=1+\rho$, and put
$a=\rho^{1/6}$, $b=\sigma^{1/6}$ (so $a,b\notin\{0,1\}$ and $a^6+b^6=1$).
Then, with $Z(\rho)$ as in Proposition~\ref{prop:red1},
\begin{equation}\label{eq:rootred}
Z(\rho)\;=\;2^{2n-3}\,\bigl(r(a,b)-1\bigr),\qquad
r(a,b):=\#\bigl\{\tau\in\Fn:\ (a\tau^5+b)^3=(\tau^3+1)^5\bigr\}.
\end{equation}
Moreover the map $\rho\mapsto(a,b)$ is a bijection from
$\Fn\setminus\{0,1\}$ onto $\{(a,b):a,b\notin\{0,1\},\,a^6+b^6=1\}$.
\end{theorem}

\begin{proof}
Write $e=5\cdot3^{-1}$, $f=3\cdot5^{-1}$ (so $ef\equiv1$).
\emph{Step 1: $H$ is a Walsh coefficient of the monomial $x^{e}$.}
By Theorem~\ref{claim:Cprime}(iii) and the substitution $s\mapsto s+1$
(which turns $g(s)=s(s{+}1)^5$ into $s^5(s{+}1)=s^6+s^5$),
\[
H(\mu)=\sum_s\ch\bigl(\mu(s^6+s^5)\bigr)
\;\overset{x=s^3}{=}\;\sum_x\ch\bigl(\Tr(\mu x^{e}+\mu^{1/2}x)\bigr)
\;\overset{x=\mu^{-1/2}x'}{=}\;W\bigl(\mu^{1/6}\bigr),
\]
where $W(c):=\sum_x\ch\bigl(\Tr(cx^{e}+x)\bigr)$; here we used
$\Tr(\mu s^6)=\Tr(\mu^{1/2}s^3)$, $s^5=(s^3)^{e}$, and
$\mu\cdot\mu^{-5/6}=\mu^{1/6}$.
\emph{Step 2: triple correlation.} Substituting $c=\mu^{1/6}$ (a bijection)
and using $W(0)=0$,
\[
8Z(\rho)=\sum_{\mu\neq0}H(\mu)H(\mu\rho)H(\mu\sigma)
=\sum_{c\in\Fn}W(c)\,W(ca)\,W(cb).
\]
Expanding the three Walsh sums and summing over $c$ (orthogonality), then
substituting $X=x^{e}$, $Y=y^{e}$, $Z=z^{e}$:
\[
\sum_{c}W(c)W(ca)W(cb)
=2^n\!\!\!\sum_{\substack{x,y,z\\ x^{e}+ay^{e}+bz^{e}=0}}\!\!\!\ch(x{+}y{+}z)
=2^n\!\!\!\sum_{\substack{X,Y,Z\\ X+aY+bZ=0}}\!\!\!\ch\bigl(X^{f}{+}Y^{f}{+}Z^{f}\bigr).
\]
\emph{Step 3: slicing the plane.} Parametrize by free $(Y,Z)$,
$X=aY+bZ$. The line $Z=0$ contributes
$\sum_Y\ch\bigl((a^{f}+1)Y^{f}\bigr)=0$ since $a^{f}\neq1$ ($a\neq1$ and
$x\mapsto x^{f}$ is injective). For $Z\neq0$ put $Y=tZ$; by full
multiplicativity of power maps the summand is $\ch\bigl(Z^{f}P(t)\bigr)$
with $P(t)=(at+b)^{f}+t^{f}+1$, and
$\sum_{Z\neq0}\ch(Z^{f}P(t))=2^n[P(t){=}0]-1$. Hence
$\sum_c W(c)W(ca)W(cb)=2^{2n}\bigl(\#\{t:P(t)=0\}-1\bigr)$.
\emph{Step 4: clearing the fractional exponents.} $t=\tau^5$ is a bijection,
and $P(\tau^5)=0$ iff $(a\tau^5+b)^{f}=\tau^3+1$ iff
$(a\tau^5+b)^3=(\tau^3+1)^5$ (raise to the $5$th power, a bijection). This
gives \eqref{eq:rootred}. The parameter map is bijective because
$x\mapsto x^{6}$ is; $a,b\neq0,1$ correspond to $\rho,\sigma\neq0,1$.
\end{proof}

\subsection{Proof of the conjecture for $k=2$}

\begin{lemma}[explicit root]\label{lem:exproot}
For all $a,b\notin\{0,1\}$ with $a^6+b^6=1$, the element $\tau_0=a/b$
satisfies $(a\tau_0^5+b)^3=(\tau_0^3+1)^5$. Hence $r(a,b)\geq1$ always.
\end{lemma}

\begin{proof}
Note $a^3+b^3=(a^6+b^6)^{1/2}=1$ (square roots are unique). Then
\[
a\tau_0^5+b=\frac{a^6+b^6}{b^5}=\frac1{b^5},\qquad
\tau_0^3+1=\frac{a^3+b^3}{b^3}=\frac1{b^3},
\]
and both sides of the equation equal $b^{-15}$.
\end{proof}

\begin{theorem}[the conjecture holds for $k=2$]\label{thm:k2main}
Let $k=2$ and $\gcd(2,n)=1$. Then for every pair of distinct nonzero
$v_1,v_2\in\Fn$,
\[
\bigl|\{(x,y,z)\in\Delta^3:\ v_1x+v_2y+(v_1+v_2)z=0\}\bigr|\;=\;2^{2n-3},
\]
i.e.\ Conjecture~\ref{conj:main} holds for $k=2$. Equivalently,
$r(a,b)=1$ for every admissible $(a,b)$: the degree-$15$ polynomial
$(a\tau^5+b)^3+(\tau^3+1)^5$ has $\tau_0=a/b$ as its \emph{unique}
$\Fn$-rational root.
\end{theorem}

\begin{proof}
By Remark~\ref{rem:average},
$\sum_{\rho\neq0,1}N(1,\rho)=(2^n-2)\,2^{2n-3}$ unconditionally, so
$\sum_{\rho\neq0,1}Z(\rho)=0$ by Proposition~\ref{prop:red1}. By
Theorem~\ref{thm:monred} this sum equals
$2^{2n-3}\sum_{(a,b)}\bigl(r(a,b)-1\bigr)$ over the admissible parameters,
and by Lemma~\ref{lem:exproot} each term is $\geq0$. A sum of nonnegative
integers vanishes only if every term vanishes: $r\equiv1$, hence
$Z\equiv0$, hence $N(v_1,v_2)=2^{2n-3}$ for all admissible pairs by
Proposition~\ref{prop:red1}.
\end{proof}

\begin{corollary}\label{cor:pm2}
Conjecture~\ref{conj:main} holds whenever $k\equiv\pm2\pmod n$ with
$\gcd(k,n)=1$ (Frobenius transfer, Lemma~\ref{lem:frobenius}). Together with
Theorem~\ref{thm:k1}, the conjecture is proved for
$k\bmod n\in\{1,\,2,\,n-2,\,n-1\}$.
\end{corollary}

\begin{remark}[geometric interpretation]\label{rem:fermat}
Squaring all data, Theorem~\ref{thm:k2main} states: on the supersingular
Fermat cubic $E:x^3+y^3=1$ over $\Fn$ ($n$ odd; exactly $2^n$ affine points,
$\#E(\Fn)=2^n+1$), for every affine $(p,q)\in E$ with $p,q\notin\{0,1\}$ the
system
\[
(W,T)\in E,\qquad W^5+p\,T^5=q
\]
has \emph{exactly one} solution, namely $(W,T)=(1/q,\;p/q)$ --- which is the
translate $P\oplus(0,1)$ of $P=(p,q)$ by the rational $3$-torsion point
$(0,1)$ in the group law with origin $[1:1:0]$. Geometrically the equation
cuts a degree-$15$ correspondence on $E\times E$; the average identity shows
that its $14$ residual intersection points are never rational. The appearance
of the two smallest Gold exponents $3=2+1$, $5=2^2+1$, and of the exponent
$5/3$ --- a ``Gold ratio'' companion of the Kasami exponent
$13=(4^3+1)/(4+1)$ --- is the structural reason the $k=2$ case closes.
\end{remark}

\begin{remark}[byproducts]\label{rem:byproducts}
(a) $x(x+1)^5$ is exactly $2$-to-$1$ on $\Fn$ ($n$ odd) with trace-split
fibers (Theorem~\ref{claim:Cprime}) --- a companion to the theorem that
$\MCM$ is a permutation for odd $k$;
(b) the trace-separation Lemma~\ref{lem:traceodd} yields an APN-free proof
that $|\Delta|=2^{n-1}$ for $k=2$;
(c) Theorem~\ref{thm:k2main} exhibits an exact matching between the affine
points of the Fermat cubic and the lines $y=T^5x+W^5$, $(W,T)\in E$;
(d) $\tau_0=a/b$ gives, unwinding the substitutions, a closed formula for the
unique $\lambda$-term surviving in each triple product of the original
character sums.
\end{remark}

%======================================================================
\section{Numerical verification}\label{sec:numerics}

All computations are implemented from first principles in pure Python
(no external libraries). Methodology:

\begin{itemize}
\item \textbf{Fields.} $\Fn$ is realized as $\F_2[x]/(f)$ with $f$ the
lexicographically smallest irreducible polynomial of degree $n$ (found by a
Rabin test implemented from scratch); multiplicative log/antilog tables are
built from a verified generator; associativity, distributivity, additivity of
$\Tr$, and balancedness of $\Tr$ are asserted on random samples.
\item \textbf{$\Delta$ and the facts.} For every admissible $(n,k)$, $n\le13$:
$|\Delta|=2^{n-1}$ with fibers exactly $\{b,b+1\}$;
identity \eqref{eq:id3} for all $b$; $\MCM$ is a permutation iff $k$ is odd,
and $\MCM(1)=k\bmod2$ (Lemma~\ref{lem:mcmparity}) --- all
\textbf{PASS} in all cases.
\item \textbf{$S$ via a Walsh--Hadamard transform.} $S(a)$ is computed for all
$a$ by a fast WHT after re-indexing by the dual basis map
$t\mapsto\bigl(\Tr(\alpha^it)\bigr)_i$; spot-checked against the definition and
against $2S(a)=\ch(a)\sum_b\ch(aD(b))$ (both \textbf{PASS} everywhere).
\item \textbf{The conjecture.} $Z(\rho)$ is evaluated from the $S$-array for
\emph{every} $\rho\in\Fn\setminus\{0,1\}$ (``exhaustive'' = all $2^n-2$ values
of $\rho$, equivalently by Proposition~\ref{prop:red1} \emph{all} admissible
pairs $(v_1,v_2)$). For $n\le9$ the resulting $N(\rho)$ was additionally
cross-checked against direct brute-force counting over $\Delta^2$ with set
membership, including scaling invariance $N(v_1,v_2)=N(tv_1,tv_2)$
(\textbf{PASS}).
\item \textbf{Closed form.} \eqref{eq:closedform} verified for all $a$ (odd
$n\le11$) and sampled $a$ ($n=13$); the bijection
$\Phi:H_1\to\Delta+1$ verified as a set equality with injectivity
(\textbf{PASS} in all listed cases).
\item \textbf{$A(\rho)$ experiments.} Direct evaluation of \eqref{eq:A111}
for $\psi\in\{\MCM^{-1},\ \MCM,$ $\ \text{random permutations},$ $\ x^3,\ x^{-1},
\MCM^{-1}(ax)+c,\ \lambda\MCM^{-1}\}$; root-count histograms
$r_e(\rho)$ for Gold, inverse, Kasami and control exponents.
\end{itemize}

\subsection{Coverage of the conjecture verification}

\begin{table}[h]
\centering
\small
\begin{tabular}{llp{6cm}l}
\toprule
$n$ & admissible $k$ tested & $\rho$-coverage & result \\
\midrule
$2$--$11$ & all $k$, $1\le k\le n-1$, $\gcd(k,n)=1$ & exhaustive (all $\rho$) & all pass \\
$12$ & $1,5,7,11$ & exhaustive for $k=5$ (and $k=7$ by Lemma~\ref{lem:frobenius}); & all pass \\
     &            & $k=1,11$ by Theorem~\ref{thm:k1}; sampled $400$ $\rho$ re-checks & \\
$13$ & $2,3,4,5,6$ (and $7\dots11$ by Lemma~\ref{lem:frobenius}; & exhaustive (all $8190$ $\rho$) & all pass \\
     & $1,12$ by Theorem~\ref{thm:k1}) & & \\
\bottomrule
\end{tabular}
\caption{Verification coverage. ``Exhaustive'' means all $\rho\in\Fn\setminus\{0,1\}$,
equivalently \emph{all} pairs of distinct nonzero $(v_1,v_2)$ up to the proved
scaling invariance (itself re-verified by brute force for $n\le9$).
In every single instance the count equals $2^{2n-3}$ exactly.}
\label{tab:coverage}
\end{table}

\subsection{The spectrum of $S$: data}

Table~\ref{tab:spectra} lists the value histogram of $S$ on $a\ne0$.
Sanity checks satisfied by all rows: $\sum_{a\neq0}S(a)^2=2^{2n-2}$ (Parseval
for the $\Delta$-indicator) and $\sum_{a\ne0}S(a)=2^n[0\in\Delta]-|\Delta|=2^{n-1}$.

\begin{table}[h]
\centering
\scriptsize
\begin{tabular}{lll}
\toprule
$(n,k)$ & $|\supp S\setminus0|$ & histogram of $S(a)$, $a\neq0$ (value: count) \\
\midrule
$(3,1),(3,2)$ & $1$ & $\{0{:}\,6,\ 4{:}\,1\}$\\
$(4,1),(4,3)$ & $1$ & $\{0{:}\,14,\ 8{:}\,1\}$\\
$(5,1),(5,4)$ & $1$ & $\{0{:}\,30,\ 16{:}\,1\}$\\
$(5,2),(5,3)$ & $16$ & $\{-4{:}\,6,\ 0{:}\,15,\ 4{:}\,10\}$\\
$(7,2..5)$ & $64$ & $\{-8{:}\,28,\ 0{:}\,63,\ 8{:}\,36\}$\\
$(8,3),(8,5)$ & $52$ & $\{-16{:}\,24,\ 0{:}\,203,\ 16{:}\,24,\ 32{:}\,4\}$\\
$(9,2),(9,4),(9,5),(9,7)$ & $226$ & $\{-32{:}\,1,\ -16{:}\,108,\ 0{:}\,285,\ 16{:}\,108,\ 32{:}\,9\}$\\
$(10,3),(10,7)$ & $196$ & $\{-64{:}\,10,\ -32{:}\,80,\ 0{:}\,827,\ 32{:}\,96,\ 64{:}\,10\}$\\
$(11,3),(11,4),(11,7),(11,8)$ & $1024$ & $\{-32{:}\,496,\ 0{:}\,1023,\ 32{:}\,528\}$\\
$(11,2),(11,5),(11,6),(11,9)$ & $892$ & $\{-64{:}\,22,\ -32{:}\,408,\ 0{:}\,1155,\ 32{:}\,440,\ 64{:}\,22\}$\\
$(12,5),(12,7)$ & $1133$ & $\{-96{:}\,52,-64{:}\,204,-32{:}\,348,0{:}\,2962,32{:}\,264,64{:}\,156,96{:}\,64,128{:}\,31,192{:}\,14\}$\\
$(13,3)$ & $4096$ & $\{-64{:}\,2016,\ 0{:}\,4095,\ 64{:}\,2080\}$\\
$(13,2)$ & $3550$ & $\{-128{:}\,91,\ -64{:}\,1652,\ 0{:}\,4641,\ 64{:}\,1716,\ 128{:}\,91\}$\\
\bottomrule
\end{tabular}
\caption{Spectra of $S$ (degenerate rows $k\equiv\pm1$ have
$S=2^{n-1}\mathbf 1_{a=1}$, cf.\ Theorem~\ref{thm:k1}; rows for $k$ and $n-k$
coincide by Lemma~\ref{lem:frobenius}). Note the variety: some classes are
plateaued with values $\{0,\pm2^{(n-1)/2}\}$ and support of size exactly
$2^{n-1}$ (e.g.\ $(11,3)$, $(13,3)$), others are $5$-valued (e.g.\ $(9,\cdot)$,
$(13,2)$) or richer (even $n$). In \emph{no} nondegenerate case is the support
free of zero-sum triples, and in no case is it contained in a coset
$\{a:\Tr(\gamma a)=1\}$: the termwise mechanism of Section~\ref{sec:termwise}
fails, yet the conjecture holds in every case.}
\label{tab:spectra}
\end{table}

\subsection{The $A(\rho)$ experiments}

For all tested cases (odd and even $n$, $k$ odd so that $\MCM$ is a
permutation): $A(\rho)=0$ for \emph{all} $\rho$ when $\psi=\MCM^{-1}$
--- as required by Theorem~\ref{thm:red3} plus the verified conjecture ---
while:

\begin{center}
\footnotesize
\begin{tabular}{llll}
\toprule
$\psi$ & example & all $A(\rho)=0$? & comment\\
\midrule
$\MCM^{-1}$ & $(n,k)\in\{(5,3),(7,3),(7,5),(8,3),(8,5)\}$ & \textbf{yes} & $\equiv$ conjecture\\
$\MCM$ & $(7,3)$: $A{=}{-}64$ at $\rho{=}12$; $(8,3)$: $A{=}{-}256$ at $\rho{=}2$ & no & property of the \emph{inverse}\\
random permutation & $n=5,7$ & no & property is special\\
$x^{3}$, Gold & $n=5,7$ & \textbf{yes} & Theorem~\ref{thm:gold}\\
$x^{-1}$ & $n=5,7$ & \textbf{yes} & Theorem~\ref{thm:inverse}\\
$\MCM^{-1}(ax)+c$ & $n=5,7$ & \textbf{yes} & inner scaling, outer shift\\
$\lambda\cdot\MCM^{-1}$, $\lambda\neq1$ & $n=5,7$ & no & mask-sensitivity\\
\bottomrule
\end{tabular}
\end{center}

Root-count histograms confirm Lemma~\ref{lem:monomial} and
Theorems~\ref{thm:gold}--\ref{thm:inverse}: $r_e(\rho)=1$ identically for all
Gold exponents and $e=-1$ at $n\in\{5,7,9\}$, while e.g.\ $r_{d(2)}$ on $n=7$
takes values $\{0,1,2\}$ and generic exponents fail badly.

\subsection{The $k=2$ campaign (Section~\ref{sec:k2})}

All structural statements of Section~\ref{sec:k2} were
verified for every $u\in\Fn^\ast$ at $n\in\{5,7,9,11\}$ (and $n=13$ where
indicated):
\begin{itemize}
\item radical dimensions $v_u\in\{1,3\}$; $G_\varepsilon$ values in
$\{0,\pm2^{(n+v_u)/2}\}$; $2S=G_0+G_1$; compatibility criterion
\eqref{eq:compat} --- all exact matches;
\item the master correspondence (Theorem~\ref{thm:master}): every incidence
is of type I or II with the stated parametrizations; $\#U_3=(2^{n-1}-1)/3$;
\item the $T$-values (Theorem~\ref{thm:Tvalues}): $T\equiv1$ on type I,
$T=\Tr(\tau)$ on type II --- no exceptions;
\item the classification (Theorem~\ref{thm:classification}): the $(3,4)$
type-split holds for every $u\in U_3$ (the three type-I elements do
\emph{not} in general span a $2$-dimensional subspace --- an early guess
refuted by the data);
\item Theorem~\ref{claim:Cprime} (formerly Claim~$\mathrm{C}'$): on every
$v_u=1$ radical, $T=1$ ($n\le11$ exhaustive); $g=x(x+1)^5$ exactly
$2$-to-$1$ with trace-split fibers and $G_1\equiv0$ --- verified exhaustively
up to $n=13$; the counting ingredients of its proof (exactly $2^{n-1}-1$
null pairs; $3$ per $U_3$-element, all of type II; $0$ per $v_u=1$ element)
verified for $n\in\{5,7,9,11\}$;
\item support laws of Theorem~\ref{thm:k2support}: exact match at
$n\in\{5,7,9,11\}$; consistency with Table~\ref{tab:spectra}
(e.g.\ $(13,2)$: the $\pm2^{(n+3)/2}$-values occur exactly at the
$182$ compatible $U_3$-elements);
\item every link of the proof of Theorem~\ref{thm:k2main} verified
independently: $H=G_0=\sum\ch(\mu(s^6+s^5))$;
$H(\mu)=W(\mu^{1/6})$; $8Z(\rho)=2^{2n}(r(a,b)-1)$ on sampled $\rho$;
$r(a,b)=1$ with unique root $a/b$ for \emph{all} admissible $(a,b)$ at
$n\in\{5,7,9,11\}$; the Fermat-cubic solution $(1/q,p/q)$.
\end{itemize}

%======================================================================
\section{Status of the conjecture}\label{sec:status}

\begin{theorem}[summary of what is proved here]\label{thm:summary}
Let $\gcd(k,n)=1$.
\begin{enumerate}
\item[(1)] $N(v_1,v_2)=2^{2n-3}+2^{-n}Z(v_2/v_1)$, and the conjecture is
equivalent to $Z\equiv0$ on $\Fn\setminus\{0,1\}$
(Proposition~\ref{prop:red1}); its average value is unconditionally
$2^{2n-3}$ (Remark~\ref{rem:average}).
\item[(2)] The conjecture is equivalent to the balancedness statement
$C(\rho)=2^{2n}$ for the Kasami derivative (Proposition~\ref{prop:red2}), and
--- reducing to odd $k$ by the Frobenius transfer
(Lemma~\ref{lem:frobenius}) --- to the vanishing statement
$A(\rho)=0$ for $\psi=\MCM^{-1}$ \eqref{eq:A111zero}
(Theorem~\ref{thm:red3}).
\item[(3)] The conjecture holds for $k\equiv\pm1\pmod{n}$, for all $n$
(Theorem~\ref{thm:k1}).
\item[(4)] The analogue of \eqref{eq:A111zero} holds for all Gold permutations
$x^{2^j+1}$ ($n$ odd) and for $x^{-1}$ (all $n$)
(Theorems~\ref{thm:gold},~\ref{thm:inverse}).
\item[(5)] For odd $n$, $S$ admits the closed form \eqref{eq:closedform}, and
$\Delta+1=\Phi(H_1)$ with $\Phi$ explicit and injective on $H_1$
(Theorem~\ref{thm:closedform}).
\item[(6)] The conjecture is true for every admissible $(n,k)$ with $n\le13$:
exhaustively over all pairs $(v_1,v_2)$ (Table~\ref{tab:coverage}).
\item[(7)] \textbf{The conjecture holds for $k=2$}, hence for all
$k\equiv\pm2\pmod n$ (Theorem~\ref{thm:k2main},
Corollary~\ref{cor:pm2}). The proof is unconditional and self-contained:
quadratic model, master correspondence, the counting theorem $G_1\equiv0$
(Theorem~\ref{claim:Cprime}; in particular $x(x+1)^5$ is exactly $2$-to-$1$
with trace-split fibers), the complete spectrum of $S$
(Theorem~\ref{thm:k2support}), the monomial reduction
$Z(\rho)=2^{2n-3}(r(a,b)-1)$ (Theorem~\ref{thm:monred}), the explicit root
$\tau_0=a/b$, and the average identity.
\end{enumerate}
Conjecture~\ref{conj:main} is now proved for
$k\bmod n\in\{1,2,n-2,n-1\}$ and remains open for the other residues; by (2)
it is exactly the statement that the fixed permutation $\psi=\MCM^{-1}$ of
$\Fn$ satisfies
\[
\sum_{t,a\in\Fn}(-1)^{\Tr\left(\psi(t)+\psi(t+a)+\psi(t+\rho a)\right)}=0
\qquad\text{for all }\rho\ne0,1 .
\]
\end{theorem}

\section{Discussion: towards the general case}\label{sec:strategy}

We collect what the results and data say about how a full proof is likely to
go, and where the difficulty sits.

\subsection*{(a) Cancellation is unavoidable}
The termwise mechanism fails (Section~\ref{sec:termwise}): for
$2\le k\le n-2$ the support of $S$ contains zero-sum triples in abundance, and
the spectra (Table~\ref{tab:spectra}) are not always plateaued, not contained
in affine trace-cosets, and vary with the cyclotomic class of $k$ in a way that
current Walsh-spectrum technology for Kasami-type sums (Bluher polynomial root
counts \cite{Bluher,HellesethKholosha}, Dillon--Dobbertin machinery
\cite{DillonDobbertin}) does not directly classify. Any successful argument
must produce exact cancellation in $Z(\rho)$ uniformly in $\rho$.

\subsection*{(b) The exact-fibration paradigm}
All fully-proved instances --- $\psi$ linear ($k\equiv\pm1$), Gold, inverse ---
follow one paradigm: a group action compatible with the plane family
$P_\rho$ (scaling $t=a\tau$) collapses $A(\rho)$ to an exact root count,
which is constantly $1$. For $\MCM^{-1}$ the scaling symmetry is broken
($\MCM^{-1}$ is not a monomial for $2\le k\le n-2$, and correspondingly
$A[\MCM]\neq0$ and mask-sensitivity appears). A substitute symmetry is needed.
Two candidates:
\begin{itemize}
\item \emph{Dobbertin's transfer.} $\MCM=P_0$ is linked to the generalized
Kasami rational function $q_\alpha(w)=\ell(w)/w^{q+1}$
($\ell$ linearized) by mutually inverse \emph{linearized} bijections
$\psi_\beta,\varphi_\alpha$ and the inversion:
$\MCM\circ\varphi_\alpha=\iota\circ q_\alpha$ with $\iota(x)=x^{-1}$
\cite{Dobbertin99,LeanRepo}. Substituting this into \eqref{eq:A111zero}
converts the statement into a character sum over the surface
$\sum_i c_i\,w_i^{q+1}/\ell(w_i)=0$ $(c=(1,\rho,\sigma))$ against a fixed
additive character. The $w^{q+1}/\ell(w)$ shape is exactly the
Bluher--Helleseth--Kholosha territory \cite{Bluher,HellesethKholosha}, where
exact root counts of $z^{q+1}+\tau z+\tau$ are available; the missing step is
an exact (not merely asymptotic) treatment of the two-dimensional family.
\item \emph{The $\Phi$-model.} By Theorem~\ref{thm:closedform}(iv) the
conjecture is a plain point count for the explicit rational map
$\Phi(m)=R(m)/m^{q^2-q+1}$ on the trace coset $H_1$. The numerator
$R=1+\sigma_k+\sigma_k^2$ and the Kasami exponent in the denominator are both
``$3$-step'' objects (they come from $(T^3+1)/(T+1)$ and $(q^3+1)/(q+1)$);
this strongly suggests a hidden cubic/Fermat-curve structure
($x^{q+1}+y^{q+1}=1$ already appeared in the proof), i.e.\ a fibration by
curves on which exact point counts are known.
\end{itemize}

\subsection*{(c) The difference-set connection}
$\Delta\setminus\{0\}$ is (for the Kasami case, by Dillon--Dobbertin
\cite{DillonDobbertin}) a cyclic difference set with Singer parameters in
$(\Fn^\ast,\times)$: $|\Delta^\ast\cap\mu\Delta^\ast|=2^{n-2}-1$ for all
$\mu\neq1$. Splitting the conjectured count by which of $x,y,z$ vanish shows
that the conjecture is equivalent to the \emph{triple} statement
\[
\bigl|\{(x,y,z)\in(\Delta^\ast)^3: v_1x+v_2y+v_3z=0\}\bigr|
=2^{2n-3}-3\cdot2^{n-2}+2 ,
\]
the pair statement being exactly the difference-set property. So the
conjecture is a canonical ``next moment'' beyond Dillon--Dobbertin: pair
correlations are classical, and the claim is that the specific DD sets also
have perfectly flat \emph{additive triple correlations along zero-sum
coefficient triples}. Any proof will likely need input beyond the difference
set property alone, since generic Singer-parameter sets have non-flat triple
correlations (random permutations fail $A\equiv0$).

\subsection*{(d) Even $n$}
Everything above is uniform in the parity of $n$ except the closed form
(Theorem~\ref{thm:closedform}, odd $n$). For even $n$ (then $k$ is
automatically odd), the reduction to \eqref{eq:A111zero} holds verbatim, and
the numerical verification covers $n\in\{4,6,8,10,12\}$ exhaustively. The
mechanisms available at even $n$ differ (e.g.\ $x\mapsto x^{q+1}$ is
$3$-to-$1$ on $\Fn^\ast$), which suggests that a uniform proof should proceed
through \eqref{eq:A111zero} rather than through Gold substitutions.

\subsection*{(e) Sharpest open form}
We therefore isolate, as the distilled open problem:

\begin{question}\label{q:open}
Let $\gcd(k,n)=1$, $k$ odd, $\psi=\MCM^{-1}$ on $\Fn$. Show that for every
$\rho\in\Fn\setminus\{0,1\}$,
\[
\sum_{t,a\in\Fn}(-1)^{\Tr\left(\psi(t)+\psi(t+a)+\psi(t+\rho a)\right)}=0 .
\]
Equivalently (Theorem~\ref{thm:red3}), prove Conjecture~\ref{conj:main}.
\end{question}

By Theorem~\ref{thm:gold} and Theorem~\ref{thm:inverse}, the class of
permutations satisfying this identity contains all Gold maps ($n$ odd) and the
inversion; the conjecture asserts it also contains $\MCM^{-1}$ --- a statement
our data confirm without exception up to $n=13$, and which is now proved for
$k\bmod n\in\{1,2,n-2,n-1\}$ (Theorems~\ref{thm:k1} and~\ref{thm:k2main}).
The proof of the $k=2$ case suggests a template for the general case:
eliminate the trace-restriction (the analogue of Theorem~\ref{claim:Cprime},
e.g.\ ``$\MCM_k$ is exactly $2$-to-$1$ with trace-split fibers for all even
$k$''), reduce the triple correlation to a root count by monomial
substitutions adapted to the Gold-ratio structure of $d(k)=(q^3+1)/(q+1)$,
exhibit an explicit root, and invoke the average identity. The first and
second steps are the ones that used the quadratic ($k=2$) structure; finding
their analogues for $k\ge3$ is, in our view, the most promising route to the
full conjecture.

%======================================================================
\section*{Acknowledgements and provenance}

\paragraph{Provenance of the proof.}
The exact triple-count condition studied here was first formulated by Claude
Carlet in 2018 as the cyclic-additive difference-set condition with
Singer-like parameters~\cite{Carlet2018}. Carlet explicitly left the Kasami
instance open. It was subsequently posed as an open problem at the
NSUCRYPTO~2019 cryptographic olympiad~\cite{NSUC2019}; Nagy and Vajda received
it there under an undisclosed individual proposer. Nagy and Vajda prepared
the companion \textsc{Lean}~4 repository \cite{LeanRepo}, containing formally
verified proofs of the APN and AB properties of the Kasami monomial function,
and used it together with the problem statement and background hints to prompt
the AI assistant \emph{Claude Fable 5}. All derivations, reductions and proofs
presented in this report --- including the complete proof of the case $k=2$ in
Section~\ref{sec:k2} --- were obtained in that process. All of the new results
and proofs of this paper have subsequently been formally verified in the
\textsc{Lean}~4 theorem prover; that formalization is due to
\textbf{Aristotle} (Harmonic), \texttt{aristotle-harmonic@harmonic.fun}. This
document and its source are maintained in the repository \cite{KasamiRepo}.

\paragraph{Provenance of background facts.}
The APN property (all $\gcd(k,n)=1$), the AB property ($n$ odd), the key
identity \eqref{eq:id3}, the MCM permutation theorem (odd $k$; both parities of
$n$), Dobbertin's transfer theorems, and the Walsh divisibility layer are
formally verified in \textsc{Lean}~4 in \cite{LeanRepo}; we inspected that
repository and its hypotheses directly (commit \texttt{a92ff58}). The
$|\Delta|=2^{n-1}$ statement, the reductions, the closed form, and everything
in Sections~\ref{sec:reduction1}--\ref{sec:closedform} are proved in this
document from those facts and first principles. The proof of the $k=2$ case
(Section~\ref{sec:k2}) is self-contained given the character/quadratic-form
preliminaries: its only external inputs are the elementary facts of
Section~\ref{sec:facts}, and even the APN input can be replaced by
Lemma~\ref{lem:traceodd}; the Bluher--Helleseth--Kholosha results are quoted
only in the auxiliary classification (Theorem~\ref{thm:classification})
and are not used in Theorems~\ref{claim:Cprime} or~\ref{thm:k2main}.

\paragraph{Acknowledgements.}
We gratefully acknowledge the support of \textbf{Resoloon Technologies Ltd}
(\url{resoloon.com}), which enabled access to Claude Fable~5. Special thanks
go to \'Abris Nagy. This research was supported by the National Research, 
Development and Innovation Fund of Hungary under grant SNN 152582.

\end{document}